\documentclass{amsart}
\usepackage{amssymb}
\usepackage{amsmath}
\usepackage{amsfonts}

\newtheorem{theorem}{Theorem}[section]

\newtheorem{corollary}[theorem]{Corollary}
\newtheorem{definition}[theorem]{Definition}
\newtheorem{proposition}[theorem]{Proposition}

\newtheorem{remark}[theorem]{Remark}
\newtheorem{lemma}[theorem]{Lemma}
\numberwithin{theorem}{section}
\input{tcilatex}

\begin{document}
\thanks{
This research was supported by CEEX grant-code PR-D11-PT00-48/2005 from The
Romanian Ministry of Education and Research}
\title{CROSSED\ PRODUCTS\ OF\ LOCALLY\ $C^{\ast }$-ALGEBRAS\ AND\ MORITA\
EQUIVALENCE}
\author{MARIA\ JOI\c{T}A}
\maketitle

\begin{abstract}
We introduce the notion of strong Morita equivalence for group actions on
locally $C^{\ast }$-algebras and prove that the crossed products associated
with two strongly Morita equivalent continuous inverse limit actions of a
locally compact group $G$ on the locally $C^{\ast }$-algebras $A$ and $B$
are strongly Morita equivalent. This generalizes a result of F. Combes,
Proc. London Math. Soc. 49(1984) and R. E. Curto, P.S. Muhly, D. P.
Williams, Proc. Amer. Soc. 90(1984).
\end{abstract}

\section{Introduction}

Locally $C^{\ast }$-algebras are generalizations of $C^{\ast }$-algebras.
Instead of being given by a single $C^{\ast }$-norm, the topology on a
locally $C^{\ast }$-algebra is defined by a directed family of $C^{\ast }$%
-seminorms. Such important concepts as Hilbert $C^{\ast }$-modules, Morita
equivalence, crossed products of $C^{\ast }$-algebras can be defined in the
framework of locally $C^{\ast }$-algebras and many results can be extended.
The proofs, however, are not quite as straightforward.

Phillips \cite{15} introduced the notion of inverse limit action of a
locally compact group on a metrizable locally $C^{\ast }$-algebra and he
considered the crossed products of metrizable locally $C^{\ast }$-algebras
by an inverse limit action. In \cite{8} we considered the crossed products
of locally $C^{\ast }$-algebras by inverse limit actions and using some
results about actions of Hopf $C^{\ast }$-algebras on locally $C^{\ast }$%
-algebras we proved a Takai duality theorem for crossed products of locally $%
C^{\ast }$-algebras.

The notion of strong Morita equivalence of locally $C^{\ast }$-algebras was
introduced in \cite{6}. It is well known that if $\alpha $\ and $\beta $\
are two strongly Morita equivalent actions of a locally compact group $G$\
on two $C^{\ast }$-algebras $A$\ and $B$, then the crossed products $G\times
_{\alpha }A$\ and $G\times _{\beta }B$\ are strongly Morita equivalent \cite%
{2,3}. In this work, we extend this result in the context of locally $%
C^{\ast }$-algebras.

The paper is organized as follows. In Section 2 are presented some results
about locally $C^{\ast }$-algebras \cite{4,5,6}, Hilbert modules over
locally $C^{\ast }$-algebras \cite{10,14} and crossed products of locally $%
C^{\ast }$-algebras \cite{8,15}. In Section 3, we introduce the notion of
action of a locally compact group $G$ on a Hilbert module $E$ over a locally 
$C^{\ast }$-algebra $A,$ and we show that, if $E$ is full, such an action
induces two actions of $G$ on the locally $C^{\ast }$-algebras $A$ and $K(E)$%
, the locally \ $C^{\ast }$-algebra of all compact operators on $E$.
Moreover, these actions are inverse limit actions if the action of $G$ on $E$
is an inverse limit action. The notion of strong Morita equivalence for
group actions on locally $C^{\ast }$-algebras is introduced in Section 4,
and we prove that the strong Morita equivalence of group actions on locally $%
C^{\ast }$-algebras is an equivalence relation. In Section 5, we prove that
the crossed products of two locally $C^{\ast }$-algebras by two strongly
Morita equivalent inverse limit actions are strongly Morita equivalent.

\section{Preliminaries}

A locally $C^{\ast }$-algebra is a complete Hausdorff complex topological $%
\ast $-algebra $A$ whose topology is determined by its continuous $C^{\ast }$%
-seminorms in the sense that a net $\{a_{i}\}_{i\in I}$ converges to $0$ in $%
A$ if and only if the net $\{p(a_{i})\}_{i\in I}$ converges to $0$ for all
continuous $C^{\ast }$-seminorms $p$ on $A$. The term of \textquotedblright
locally $C^{\ast }$-algebra\textquotedblright\ is due to Inoue. In the
literature, locally $C^{\ast }$-algebras have been given different names
such as $b^{\ast }$-algebras ( C. Apostol ), $LMC^{\ast }$-algebras ( G.
Lassner, K. Schm\"{u}dgen) or pro- $C^{\ast }$-algebras (N. C. Phillips).

The set $S(A)$ of all continuous $C^{\ast }$-seminorms on $A$ is directed
with the order $p\geq q$ if $p\left( a\right) \geq q\left( a\right) $ for
all $a\in A$. For each $p\in S(A),$ $\ker p=\{a\in A;p(a)=0\}$ is a
two-sided $\ast $-ideal of $A$ and the quotient $\ast $-algebra $A/\ker p$,
denoted by $A_{p}$, is a $C^{\ast }$-algebra in the $C^{\ast }$-norm induced
by $p$. The canonical map from $A$ to $A_{p}$ is denoted by $\pi _{p}^{A}$.

For $p,q\in S(A)$ with $p\geq q$ there is a canonical surjective morphism of 
$C^{\ast }$-algebras $\pi _{pq}^{A}:A_{p}\rightarrow A_{q}$ such that $\pi
_{pq}^{A}\circ \pi _{p}^{A}=\pi _{q}^{A}.$ Then $\{A_{p};\pi
_{pq}^{A}\}_{p,q\in S(A),p\geq q}$ is an inverse system of $C^{\ast }$%
-algebras and moreover, the locally $C^{\ast }$-algebras $A$ and $%
\lim\limits_{\underset{p\in S(A)}{\leftarrow }}A_{p}$ are isomorphic.

A morphism of locally $C^{\ast }$-algebras is a continuous morphism of $\ast 
$ -algebras. Two locally $C^{\ast }$-algebras $A$ and $B$ are isomorphic if
there is a bijective map $\Phi :A\rightarrow B$ such that $\Phi $ and $\Phi
^{-1}$ are morphisms of locally $C^{\ast }$-algebras.

A representation of a locally $C^{\ast }$-algebra $A$ on a Hilbert space $H$
is a continuous $\ast $-morphism $\varphi $ from $A$ to $L(H)$, the $C^{\ast
}$-algebra of all bounded linear operators on $H$. We say that the
representation $\varphi $ is non-degenerate if $\varphi \left( A\right) H$
is dense in $H.$

Hilbert modules over locally $C^{\ast }$-algebras are generalizations of
Hilbert $C^{\ast }$-modules by allowing the inner product to take values in
a locally $C^{\ast }$-algebra rather than in a $C^{\ast }$-algebra.

\begin{definition}
A pre-Hilbert$\ A$-module is a complex vector space$\ E$\ which is also a
right $A$-module, compatible with the complex algebra structure, equipped
with an $A$-valued inner product $\left\langle \cdot ,\cdot \right\rangle
:E\times E\rightarrow A\;$which is $\mathbb{C}$ -and $A$-linear in its
second variable and satisfies the following relations:

\begin{enumerate}
\item $\left\langle \xi ,\eta \right\rangle ^{\ast }=\left\langle \eta ,\xi
\right\rangle \;\;$for every $\xi ,\eta \in E;$

\item $\left\langle \xi ,\xi \right\rangle \geq 0\;\;$for every $\xi \in E;$

\item $\left\langle \xi ,\xi \right\rangle \geq 0\;\;$for every $\xi \in
E;\left\langle \xi ,\xi \right\rangle =0\;$\ if and only if $\xi =0.$
\end{enumerate}

We say that $E\;$is a Hilbert $A$-module if $E\;$is complete with respect to
the topology determined by the family of seminorms $\{\overline{p}%
_{E}\}_{p\in S(A)}$,$\;$where  $\overline{p}_{E}(\xi )=\sqrt{p\left(
\left\langle \xi ,\xi \right\rangle \right) }$, $\xi \in E$.\smallskip 
\end{definition}

Any locally $C^{\ast }$-algebra $A$ is a Hilbert $A$ -module in a natural
way.

A Hilbert $A$-module $E$ is full if the linear space $\left\langle
E,E\right\rangle \;$generated by $\{\left\langle \xi ,\eta \right\rangle
,\;\xi ,\eta \in E\}$ is dense in $A$.

Let $E\;$be a Hilbert $A$-module.\ For $p\in S(A),\;\ker \overline{p}%
_{E}=\{\xi \in E;\overline{p}_{E}(\xi )=0\}\;$is a closed submodule of $E\;$%
and $E_{p}=E/\ker \overline{p}_{E}\;$is a Hilbert $A_{p}$-module with $(\xi
+\ker \overline{p}_{E}{})\pi _{p}^{A}(a)=\xi a+\ker \overline{p}_{E}{}\;$and 
$\left\langle \xi +\ker \overline{p}_{E}{},\eta +\ker \overline{p}%
_{E}{}\right\rangle =\pi _{p}^{A}(\left\langle \xi ,\eta \right\rangle )$.\
The canonical map from $E\;$onto $E_{p}$ is denoted by $\sigma _{p}^{E}$.

For $p,q\in S(A)$ with $p\geq q$, there is a canonical surjective morphism
of vector spaces $\sigma _{pq}^{E}\;$from $E_{p}\;$onto $E_{q}\;$such that $%
\sigma _{pq}^{E}(\sigma _{p}^{E}(\xi ))=\sigma _{q}^{E}(\xi )\;$for all $\xi
\in E$.$\;$Then $\{E_{p};A_{p};\sigma _{pq}^{E},$ $\pi _{pq}^{A}\}_{p,q\in
S(A),p\geq q}$ is an inverse system of Hilbert $C^{\ast }$-modules in the
following sense: $\sigma _{pq}^{E}(\xi _{p}a_{p})=\sigma _{pq}^{E}(\xi
_{p})\pi _{pq}^{A}(a_{p}),\xi _{p}\in E_{p},a_{p}\in A_{p};$ $\left\langle
\sigma _{pq}^{E}(\xi _{p}),\sigma _{pq}^{E}(\eta _{p})\right\rangle =\pi
_{pq}^{A}(\left\langle \xi _{p},\eta _{p}\right\rangle ),\xi _{p},$ $\eta
_{p}\in E_{p};$ $\sigma _{pp}^{E}(\xi _{p})=\xi _{p},\;\xi _{p}\in E_{p}\;$%
and $\sigma _{qr}^{E}\circ \sigma _{pq}^{E}=\sigma _{pr}^{E}\;$if $p\geq
q\geq r,$ and $\lim\limits_{\underset{p\in S(A)}{\leftarrow }}E_{p}$\ is a
Hilbert $A$-module with the action defined by $\left( (\xi _{p})_{p}\right)
\left( (a_{p})_{p}\right) =(\xi _{p}a_{p})_{p}\ $and the inner product
defined by $\left\langle (\xi _{p})_{p},(\eta _{p})_{p}\right\rangle =\left(
\left\langle \xi _{p},\eta _{p}\right\rangle \right) _{p}$. Moreover, the
Hilbert $A$-module $E$ can be identified with\ $\lim\limits_{\underset{p\in
S(A)}{\leftarrow }}E_{p}$.

Let $E$ and $F$ be two Hilbert $A$-modules. A module morphism $%
T:E\rightarrow F$ is adjointable if there is a module morphism $T^{\ast
}:F\rightarrow E$ such that $\left\langle T\xi ,\eta \right\rangle
=\left\langle \xi ,T^{\ast }\eta \right\rangle $ for all $\xi \in E$ and for
all $\eta \in F$. If $T$ is an adjointable module morphism from $E$ to $F$,
then, for each $p\in S(A)$, there is a positive constant $M_{p}$ such that $%
\overline{p}_{F}\left( T\xi \right) \leq M_{p}\overline{p}_{E}\left( \xi
\right) $ for all $\xi \in E$.

The set of all adjointable module morphisms from $E$ to $F$ is denoted by $%
L(E,F)$ and we write $L(E)$ for $L(E,E)$. For $p\in S(A)$,\ since $T(\ker 
\overline{p}_{E})\subseteq \ker \overline{p}_{F}$ for all $T\in L(E,F)$,\ we
can define a linear map $(\pi _{p}^{A})_{\ast }:L(E,F)\rightarrow
L(E_{p},F_{p})$\ by%
\begin{equation*}
(\pi _{p}^{A})_{\ast }(T)(\sigma _{p}^{E}(\xi ))=\sigma _{p}^{F}(T(\xi )).
\end{equation*}
We consider on $L(E,F)$\ the topology defined by family of seminorms $\{%
\widetilde{p}_{L(E,F)}\}_{p\in S(A)}$, where%
\begin{equation*}
\widetilde{p}_{L(E,F)}(T)=\left\Vert (\pi _{p}^{A})_{\ast }(T)\right\Vert
_{L(E_{p},F_{p})}
\end{equation*}%
for all $T\in L(E,F)$. Thus topologized $L(E,F)$ becomes a complete locally
convex space and $L(E)\ $becomes a locally $C^{\ast }$-algebra.

For $p,q\in S(A)$ with $p\geq q$, consider the linear map $(\pi
_{pq}^{A})_{\ast }:L(E_{p},F_{p})\rightarrow L(E_{q},F_{q})$ defined by%
\begin{equation*}
(\pi _{pq}^{A})_{\ast }(T_{p})(\sigma _{q}^{E}(\xi ))=\sigma
_{pq}^{F}(T_{p}(\sigma _{p}^{E}(\xi )))
\end{equation*}%
$T_{p}\in L(E_{p},F_{p}),$ $\ \xi \in E$. Then $\{L(E_{p},F_{p}),(\pi
_{pq}^{A})_{\ast }\}_{p,q\in S(A),p\geq q}$ is an inverse system of Banach
spaces and the complete locally convex spaces $L(E,F)\ $and $\lim\limits_{%
\underset{p\in S(A)}{\leftarrow }}L(E_{p},F_{p})$ can be identified.
Moreover, the locally $C^{\ast }$-algebras $L(E)\ $and $\lim\limits_{%
\underset{p\in S(A)}{\leftarrow }}L(E_{p})$ can be identified.

For $\xi \in E$\ and $\eta \in F$\ we consider the rank one homomorphism $%
\theta _{\eta ,\xi }$\ from $E$\ into $F$\ defined by $\theta _{\eta ,\xi
}(\zeta )=\eta \left\langle \xi ,\zeta \right\rangle $.\ Clearly, $\theta
_{\eta ,\xi }\in L(E,F)$\ and $\theta _{\eta ,\xi }^{\ast }=\theta _{\xi
,\eta }$. \smallskip The linear subspace of $L(E,F)$ spanned by $\{\theta
_{\eta ,\xi };\xi \in E,\eta $ $\in $ $F\}$ is denoted by $\Theta (E,F)$,
and the closure of $\Theta (E,F)$ in $L(E,F)$ is denoted by $K\left(
E,F\right) .$ We write $K(E)$ for $K(E,E)$. Moreover, $K(E,F)$ may be
identified with $\lim\limits_{\underset{p\in S(A)}{\leftarrow }%
}K(E_{p},F_{p})$.

We say that the Hilbert $A$-modules $E$ and $F$ are unitarily equivalent if
there is a unitary element $U$ in $L(E,F)$ (that is, $U^{*}U=$id$_{E}$ and $%
UU^{*}=$id$_{F}$).

Let $G$ be a locally compact group and let $A$ be a locally $C^{\ast }$%
-algebra. The vector space of all continuous functions from $G$ to $A$ with
compact support is denoted by $C_{c}(G,A)$.

\begin{lemma}
(\cite[Lemma 3.7]{8}] ) Let $f\in C_{c}(G,A)$. Then there is a unique
element $\int\limits_{G}f(s)ds$\ in $A$\ such that for any non-degenerate $%
\ast $-representation $(\varphi ,H_{\varphi })$\ of $A$\ 
\begin{equation*}
\left\langle \varphi (\int\limits_{G}f(s)ds)\xi ,\eta \right\rangle
=\int\limits_{G}\left\langle \varphi (f(s))\xi ,\eta \right\rangle ds
\end{equation*}%
for all $\xi ,\eta $\ in $H_{\varphi }$. Moreover, we have:

$\;\;\;$(1) $p(\int\limits_{G}f(s)ds)\leq M_{f}\sup \{p(f(s));$\ $s\in $supp$%
\left( f\right) \}$\ for some positive constant $M_{f}$\ and for all $p\in
S(A)$;

(2) $(\int\limits_{G}f(s)ds)a=\int\limits_{G}f(s)ads$\ for all $a$\ $\in A$;

(3) $\Phi (\int\limits_{G}f(s)ds)=\int\limits_{G}\Phi \left( f(s)\right) ds$
\ for any morphism of locally $C^{\ast }$-algebras $\Phi :A\rightarrow B$;

(4) $(\int\limits_{G}f(s)ds)^{\ast }=\int\limits_{G}f(s)^{\ast }ds$.
\end{lemma}

\smallskip An action of $G$\ on $A$\ is a morphism of groups $\alpha $ from $%
G$\ to Aut$\left( A\right) $, the set of all isomorphisms of locally $%
C^{\ast }$-algebras from $A$\ to $A$. The action $\alpha $\ is continuous if
the function $t\rightarrow \alpha _{t}(a)$\ from $G$\ to $A$\ is continuous
for each $a\in A$.\smallskip

An action $\alpha $\ of $G$ on $A$ is an inverse limit action if we can
write $A$\ as inverse limit $\lim\limits_{\underset{\lambda \in \Lambda }{%
\leftarrow }}A_{\lambda }$ of $C^{\ast }$-algebras in such a way that there
are actions $\alpha ^{\lambda }$\ of $G$\ on $A_{\lambda },$ $\lambda \in
\Lambda $ \ such that $\alpha _{t}=\lim\limits_{\underset{\lambda \in
\Lambda }{\leftarrow }}\alpha _{t}^{\lambda }$\ for all $t$\ in $G$\ \cite[%
Definition 5.1]{15}.

\smallskip The action $\alpha $\ of $G$\ on $A$\ is a continuous inverse
limit action if there is a cofinal subset of $G$-invariant continuous $%
C^{\ast }$-seminorms on $A$\ ( a continuous $C^{\ast }$-seminorm $p$\ on $A$%
\ is $G$\ -invariant if $p(\alpha _{t}(a))=p(a)$\ for all $a$\ in $A$\ and
for all $t$\ in $G$). Thus, for a continuous inverse limit action $%
t\rightarrow \alpha _{t}$ of $G$ on $A$ we can suppose that for each $p\in
S(A)$, there is a continuous action $t\rightarrow \alpha _{t}^{p}$ of $G$ on 
$A_{p}$ such that $\alpha _{t}=\lim\limits_{\underset{p\in S(A)}{\leftarrow }%
}\alpha _{t}^{p}$ for all $t\in G$.

Let $f,h\in C_{c}(G,A)$. The map $(s,t)\rightarrow f(t)\alpha _{t}\left(
h(t^{-1}s)\right) $ from $G\times G$ to $A$ is an element in $C_{c}(G\times
G,A)$ and the relation 
\begin{equation*}
\left( f\times h\right) \left( s\right) =\int\limits_{G}f(t)\alpha
_{t}\left( h(t^{-1}s)\right) dt
\end{equation*}%
defines an element in $C_{c}(G,A)$, called the convolution of $f$ and $h$.

The vector space $C_{c}(G,A)$ becomes a $\ast $-algebra with convolution as
product and involution defined by 
\begin{equation*}
f^{\sharp }(t)=\Delta (t)^{-1}\alpha _{t}\left( f(t^{-1})^{\ast }\right)
\end{equation*}%
where $\Delta $ is the modular function on $G$.

For any $p\in S(A)$, the map $N_{p}:C_{c}(G,A)\rightarrow $ $[0,\infty )$
defined by 
\begin{equation*}
N_{p}(f)=\int\limits_{G}p(f(s))ds
\end{equation*}%
is a submultiplicative $\ast $-seminorm on $C_{c}(G,A)$.

Let $L^{1}(G,A,\alpha )$ be the Hausdorff completion of $C_{c}(G,A)$ with
respect to the topology defined by the family of submultiplicative $\ast $%
-seminorms $\{N_{p}\}_{p\in S(A)}$. Then $L^{1}(G,A,\alpha )$ is a complete
locally $m$-convex $\ast $-algebra.

For each $p\in S(A)$ the map $n_{p}:L^{1}(G,A,\alpha )$ $\rightarrow $ $%
[0,\infty )$ defined by 
\begin{equation*}
n_{p}(f)=\sup \{\left\Vert \varphi \left( f\right) \right\Vert ;\varphi \in 
\mathcal{R}_{p}(L^{1}(G,A,\alpha )\},
\end{equation*}%
where $\mathcal{R}_{p}(L^{1}(G,A,\alpha )$ denotes the set of all
non-degenerate representations $\varphi $ of $L^{1}(G,A,\alpha )$ on Hilbert
spaces which verify the relation 
\begin{equation*}
\left\Vert \varphi \left( f\right) \right\Vert \leq N_{p}(f)
\end{equation*}%
for all $f\in L^{1}(G,A,\alpha )$, is a $C^{\ast }$-seminorm on $%
L^{1}(G,A,\alpha )$. The Hausdorff completion of $L^{1}(G,A,\alpha )$ with
respect to the topology determined by the family of $C^{\ast }$-seminorms $%
\{n_{p}\}_{p\in S(A)}$ is a locally $C^{\ast }$-algebra, denoted by $G\times
_{\alpha }A$, and called the crossed product of $A$ by the action $\alpha $.
Moreover,\smallskip 
\begin{equation*}
A\times _{\alpha }G=\lim\limits_{\overset{\longleftarrow }{p\in S(A)}%
}G\times _{\alpha ^{p}}A_{p}
\end{equation*}%
up to an isomorphism of locally $C^{\ast }$-algebras.

\section{Group actions on Hilbert modules}

Let $A$ and $B$ be two locally $C^{*}$-algebras, let $E$ be a Hilbert $A$%
-module and let $F$ be a Hilbert $B$-module.

\begin{definition}
A map $u:E\rightarrow F$ is called morphism of Hilbert modules if there is a
morphism of locally $C^{\ast }$-algebras $\alpha :A\rightarrow B$ such that 
\begin{equation*}
\left\langle u\left( \xi \right) ,u\left( \eta \right) \right\rangle =\alpha
\left( \left\langle \xi ,\eta \right\rangle \right)
\end{equation*}%
for all $\xi ,\eta \in E$. If we want to specify the morphism $\alpha $ of
locally $C^{\ast }$-algebras, we say that $u$ is an $\alpha $-morphism.
\end{definition}

\begin{remark}
Let $u:E\rightarrow F$ be an $\alpha $-morphism. Then:

\begin{enumerate}
\item $u$ is linear;

\item $u$ is continuous;

\item $u\left( \xi a\right) =u\left( \xi \right) \alpha \left( a\right) $
for all $a\in A$ and for all $\xi \in E$;

\item if $\alpha $ is injective, then $u$ is injective.
\end{enumerate}
\end{remark}

\begin{definition}
An isomorphism of Hilbert modules is a bijective map $u:E\rightarrow F$ such
that $u$ and $u^{-1}$ are morphisms of Hilbert modules.
\end{definition}

\begin{proposition}
Let $E$ be a Hilbert $A$-module, let $F$ be a Hilbert $B$-module and let $%
u:E\rightarrow F$ be an $\alpha $-morphism of Hilbert modules. If $E$ and $F$
are full and $u$ is an isomorphism of Hilbert modules, then $\alpha $ is an
isomorphism of locally $C^{\ast }$-algebras.
\end{proposition}

\begin{proof}
Indeed, since $u^{-1}:F\rightarrow E$ is a morphism of Hilbert modules,
there is a morphism of locally $C^{\ast }$-algebras $\beta :B\rightarrow A$
such that 
\begin{equation*}
\beta \left( \left\langle \eta _{1},\eta _{2}\right\rangle \right)
=\left\langle u^{-1}\left( \eta _{1}\right) ,u^{-1}\left( \eta _{2}\right)
\right\rangle
\end{equation*}%
for all $\eta _{1},\eta _{2}\in F$. Then we have: 
\begin{equation*}
\left( \alpha \circ \beta \right) \left( \left\langle \eta _{1},\eta
_{2}\right\rangle \right) =\alpha \left( \left\langle u^{-1}\left( \eta
_{1}\right) ,u^{-1}\left( \eta _{2}\right) \right\rangle \right)
=\left\langle \eta _{1},\eta _{2}\right\rangle
\end{equation*}%
for all $\eta _{1},\eta _{2}\in F$ and 
\begin{equation*}
\left( \beta \circ \alpha \right) \left( \left\langle \xi _{1},\xi
_{2}\right\rangle \right) =\beta \left( \left\langle u\left( \xi _{1}\right)
,u\left( \xi _{2}\right) \right\rangle \right) =\left\langle \xi _{1},\xi
_{2}\right\rangle
\end{equation*}%
for all $\xi _{1},\xi _{2}\in E$. From these facts and taking into account
that $E$ and $F$ are full and the maps $\alpha $ and $\beta $ are
continuous, we conclude that $\alpha \circ \beta =$id$_{B}$ and $\beta \circ
\alpha =$id$_{A}$. Therefore $\alpha $ is an isomorphism of locally $C^{\ast
}$-algebras.
\end{proof}

For a Hilbert $A$-module $E,$ 
\begin{equation*}
\text{Aut}(E)=\{u:E\rightarrow E;u\text{ is an isomorphism of Hilbert
modules }\}
\end{equation*}
is a group.

\begin{definition}
Let $G$ be a locally compact group. An action of $G$ on $E$ is a morphism of
groups $g\rightarrow u_{g}$ from $G$ to Aut$(E)$.

The action $g\rightarrow u_{g}$ of $G$ on $E$ is continuous if the map $%
g\rightarrow u_{g}\left( \xi \right) $ from $G$ to $E$ is continuous for
each $\xi \in E$.

An action $g\rightarrow u_{g}$ of $G$ on $E$ is an inverse limit action if
we can write $E$ as an inverse limit $\lim\limits_{\underset{\lambda \in
\Lambda }{\leftarrow }}E_{\lambda }$ of Hilbert $C^{\ast }$-modules in such
a way that there are actions $g\rightarrow u_{g}^{\lambda }$ of $G$ on $%
E_{\lambda }$, $\lambda \in \Lambda $ such that $u_{g}=\lim\limits_{\underset%
{\lambda \in \Lambda }{\leftarrow }}u_{g}^{\lambda }$ for each $g\in G$.
\end{definition}

\begin{remark}
Suppose that $g\rightarrow u_{g}$ is an inverse limit action of $G$ on $E$.
Then $E=\lim\limits_{\underset{\lambda \in \Lambda }{\leftarrow }}E_{\lambda
}$ and $u_{g}=\lim\limits_{\underset{\lambda \in \Lambda }{\leftarrow }%
}u_{g}^{\lambda }$ for each $g\in G$, where $g\rightarrow u_{g}^{\lambda }$
is an action of $G$ on $E_{\lambda }$ for each $\lambda \in \Lambda $.

Let $\lambda \in \Lambda $. Since $g\rightarrow u_{g}^{\lambda }$ is an
action of $G$ on $E_{\lambda }$,%
\begin{equation*}
\left\Vert u_{g}^{\lambda }\left( \sigma _{\lambda }(\xi \right) \right\Vert
_{E_{\lambda }}=\left\Vert \sigma _{\lambda }(\xi \right\Vert _{E_{\lambda }}
\end{equation*}%
for each $\xi \in E$, and for all $g\in G$ \cite[pp. 292]{2}. This implies
that 
\begin{equation*}
\overline{p}_{\lambda }(u_{g}(\xi ))=\overline{p}_{\lambda }(\xi )
\end{equation*}%
for all $g\in G$ and for all $\xi \in E$.

From these facts, we conclude that $g\rightarrow u_{g}$ is an inverse limit
action of $G$ on $E$, if $S(G,A)=\{p\in S(A);\overline{p}_{E}\left(
u_{g}\left( \xi \right) \right) =\overline{p}_{E}\left( \xi \right) $ for
all $g\in G$ and for all $\xi \in E\}$ is a cofinal subset of $S(A)$.
Therefore, if $g\rightarrow u_{g}$ is an inverse limit action of $G$ on $E$,
we can suppose that $u_{g}=\lim\limits_{\underset{p\in S(A)}{\leftarrow }%
}u_{g}^{p}$.
\end{remark}

\begin{remark}
Let $g\rightarrow u_{g}$ be an inverse limit action of $G$ on $E$. By Remark
3.6, we can suppose that $u_{g}=\lim\limits_{\underset{p\in S(A)}{\leftarrow 
}}u_{g}^{p}$. If the actions $g\rightarrow u_{g}^{p}$ of $G$ on $E_{p,}$ $%
p\in S(A)$ are all continuous, then, clearly, the action $g\rightarrow u_{g}$
of $G$ on $E$ is continuous.

Conversely, suppose that the action $g\rightarrow u_{g}$\ of $G$\ on $E$\ is
continuous. Let $p\in S(A),$\ $g_{0}\in G,$\ $\xi _{0}\in E$\ and $%
\varepsilon >0$.\ Since the map $g\rightarrow u_{g}\left( \xi _{0}\right) $
from $G$ to $E$\ is continuous, there is a neighborhood $U_{0}$\ of $g_{0}$\
such that\textbf{\ }%
\begin{equation*}
\overline{p}_{E}\left( u_{g}\left( \xi _{0}\right) -u_{g_{0}}\left( \xi
_{0}\right) \right) \leq \varepsilon 
\end{equation*}%
for all $g\in U_{0}$.\ Then%
\begin{eqnarray*}
\left\Vert u_{g}^{p}\left( \sigma _{p}^{E}(\xi _{0})\right)
-u_{g_{0}}^{p}\left( \sigma _{p}^{E}(\xi _{0})\right) \right\Vert _{E_{p}}
&=&\left\Vert \sigma _{p}^{E}\left( u_{g}(\xi _{0})\right) -\sigma
_{p}^{E}\left( u_{g_{0}}(\xi _{0})\right) \right\Vert _{E_{p}} \\
&=&\overline{p}_{E}\left( u_{g}\left( \xi _{0}\right) -u_{g_{0}}\left( \xi
_{0}\right) \right) \leq \varepsilon 
\end{eqnarray*}%
for all $g\in U_{0}$.\ This shows that the action $g\rightarrow u_{g}^{p}$\
of $G$ on $E_{p}$ is continuous. Thus we showed that the inverse limit
action $g\rightarrow u_{g}$\ of $G$\ on $E$\ is continuous if and only if
the actions $g\rightarrow u_{g}^{p}$\ of $G$\ on $E_{p,}$\ $p\in S(A)$\ are
all continuous.
\end{remark}

\begin{proposition}
Let $G$ be a locally compact group and let $E$ be a full Hilbert $A$-module.
Any action $g\rightarrow u_{g}$ of $G$ on $E$ induces an action $%
g\rightarrow \alpha _{g}^{u}$ of $G$ on $A$ such that 
\begin{equation*}
\alpha _{g}^{u}\left( \left\langle \xi ,\eta \right\rangle \right)
=\left\langle u_{g}\left( \xi \right) ,u_{g}\left( \eta \right) \right\rangle
\end{equation*}%
for all $g\in G$ and for all $\xi ,\eta \in E$ and an action $g\rightarrow
\beta _{g}^{u}$ of $G$ on $K(E)$ such that 
\begin{equation*}
\beta _{g}^{u}\left( \theta _{\xi ,\eta }\right) =\theta _{u_{g}\left( \xi
\right) ,u_{g}\left( \eta \right) }
\end{equation*}%
for all $g\in G$ and for all $\xi ,\eta \in E$. Moreover, if $g\rightarrow
u_{g}$ is a continuous inverse limit action of $G$ on $E$, then the actions
of $G$ on $A$ respectively $K(E)$ induced by $u$ are continuous inverse
limit actions.
\end{proposition}

\begin{proof}
Let $g\in G$. Since $E$ is full and $u_{g}\in $Aut$(E)$, there is an
isomorphism of locally $C^{\ast }$-algebras $\alpha _{g}^{u}:A\rightarrow A$
such that 
\begin{equation*}
\alpha _{g}^{u}\left( \left\langle \xi ,\eta \right\rangle \right)
=\left\langle u_{g}\left( \xi \right) ,u_{g}\left( \eta \right) \right\rangle
\end{equation*}%
for all $\xi ,\eta \in E$. It is not difficult to check that the map $%
g\rightarrow \alpha _{g}^{u}$ from $G$ to Aut$\left( A\right) $ is a
morphism of groups. Therefore $g\rightarrow \alpha _{g}^{u}$ is an action of 
$G$ on $A.$

Let $g\in G$. Consider the linear map $\beta _{g}^{u}:\Theta \left( E\right)
\rightarrow \Theta (E)$ defined by 
\begin{equation*}
\beta _{g}^{u}\left( \theta _{\xi ,\eta }\right) =\theta _{u_{g}\left( \xi
\right) ,u_{g}\left( \eta \right) }.
\end{equation*}%
It is not difficult to check that $\beta _{g}^{u}$ is a $\ast $-morphism.

Let $p\in S(A)$. Then 
\begin{eqnarray*}
\widetilde{p}_{L(E)}\left( \beta _{g}^{u}\left( \theta _{\xi ,\eta }\right)
\right) &=&\sup \{\overline{p}_{E}\left( u_{g}\left( \xi \right)
\left\langle u_{g}\left( \eta \right) ,u_{g}\left( \zeta \right)
\right\rangle \right) ;\overline{p}_{E}\left( u_{g}\left( \zeta \right)
\right) \leq 1\} \\
&=&\sup \{\overline{p}_{E}\left( u_{g}\left( \xi \left\langle \eta ,\zeta
\right\rangle \right) \right) ;\overline{p}_{E}\left( u_{g}\left( \zeta
\right) \right) \leq 1\} \\
&=&\sup \{\left( p\circ \alpha _{g}^{u}\right) \left( \left\langle \zeta
,\eta \right\rangle \left\langle \xi ,\xi \right\rangle \left\langle \eta
,\zeta \right\rangle \right) ^{1/2};\left( p\circ \alpha _{g}^{u}\right)
\left( \zeta \right) \leq 1\} \\
&=&\sup \{r\left( \left\langle \zeta ,\eta \right\rangle \left\langle \xi
,\xi \right\rangle \left\langle \eta ,\zeta \right\rangle \right)
^{1/2};r\left( \zeta \right) \leq 1\} \\
&=&\widetilde{r}_{L(E)}\left( \theta _{\xi ,\eta }\right)
\end{eqnarray*}%
where $r=p\circ \alpha _{g}^{u}\in S(A)$, for all $\xi ,\eta \in E$. From
this fact, we conclude that $\beta _{g}^{u}$ extends to a morphism from $%
K(E) $ to $K(E)$, denoted also by $\beta _{g}^{u}$. Moreover, since 
\begin{equation*}
\left( \beta _{g}^{u}\circ \beta _{g^{-1}}^{u}\right) \left( \theta _{\xi
,\eta }\right) =\theta _{\xi ,\eta }=\left( \beta _{g^{-1}}^{u}\circ \beta
_{g}^{u}\right) \left( \theta _{\xi ,\eta }\right)
\end{equation*}%
for all $\xi ,\eta \in E$, $\beta _{g}^{u}$ is invertible and $\left( \beta
_{g}^{u}\right) ^{-1}=$ $\beta _{g^{-1}}^{u}$. Therefore $\beta _{g}^{u}\in $%
Aut$\left( K(E)\right) $. It is not difficult to check that the map $%
g\rightarrow $ $\beta _{g}^{u}$ is a morphism of groups and so it defines an
action of $G$ on $K(E).$

Now suppose that $g\rightarrow u_{g}$ is a continuous inverse limit action
of $G$ on $E$. Let $p\in S(A)$. Then the map $g\rightarrow u_{g}^{p}$ is a
continuous action of $G$ on $E_{p}$ and so it induces a continuous action $%
g\rightarrow \alpha _{g}^{u^{p}}$ of $G$ on $A_{p}$ such that 
\begin{equation*}
\alpha _{g}^{u^{p}}\left( \left\langle \sigma _{p}^{E}(\xi ),\sigma
_{p}^{E}(\eta )\right\rangle \right) =\left\langle u_{g}^{p}\left( \sigma
_{p}^{E}(\xi )\right) ,u_{g}^{p}\left( \sigma _{p}^{E}(\eta )\right)
\right\rangle
\end{equation*}%
for all $\xi ,\eta \in E$ and for all $g\in G$ and a continuous action $%
g\rightarrow \beta _{g}^{u^{p}}$ of $G$ on $K(E_{p})$ such that 
\begin{equation*}
\beta _{g}^{u^{p}}\left( \theta _{\sigma _{p}^{E}(\xi ),\sigma _{p}^{E}(\eta
)}\right) =\theta _{u_{g}^{p}\left( \sigma _{p}^{E}(\xi )\right)
,u_{g}^{p}\left( \sigma _{p}^{E}(\eta )\right) }
\end{equation*}%
for all $g\in G$ and for all $\xi ,\eta \in E\ $( see, for example, \cite{2}%
).

Let $p,q\in S(A)$ with $p\geq q$ and $g\in G$. Then:%
\begin{eqnarray*}
\left( \pi _{pq}^{A}\circ \alpha _{g}^{u^{p}}\right) \left( \left\langle
\sigma _{p}^{E}(\xi ),\sigma _{p}^{E}(\eta )\right\rangle \right) &=&\pi
_{pq}^{A}\left( \left\langle u_{g}^{p}\left( \sigma _{p}^{E}(\xi )\right)
,u_{g}^{p}\left( \sigma _{p}^{E}(\eta )\right) \right\rangle \right) \\
&=&\left\langle \left( \sigma _{pq}^{E}\circ u_{g}^{p}\right) \left( \sigma
_{p}^{E}(\xi )\right) ,\left( \sigma _{pq}^{E}\circ u_{g}^{p}\right) \left(
\sigma _{p}^{E}(\eta )\right) \right\rangle \\
&=&\left\langle u_{g}^{q}\left( \sigma _{q}^{E}(\xi )\right)
,u_{g}^{q}\left( \sigma _{q}^{E}(\eta )\right) \right\rangle \\
&=&\alpha _{g}^{u^{q}}\left( \left\langle \sigma _{q}^{E}(\xi ),\sigma
_{q}^{E}(\eta )\right\rangle \right) \\
&=&\left( \alpha _{g}^{u^{q}}\circ \pi _{pq}^{A}\right) \left( \left\langle
\sigma _{p}^{E}(\xi ),\sigma _{p}^{E}(\eta )\right\rangle \right)
\end{eqnarray*}%
and 
\begin{eqnarray*}
\left( \left( \pi _{pq}^{A}\right) _{\ast }\circ \beta _{g}^{u^{p}}\right)
\left( \theta _{\sigma _{p}^{E}(\xi ),\sigma _{p}^{E}(\eta )}\right)
&=&\left( \pi _{pq}^{A}\right) _{\ast }\left( \theta _{u_{g}^{p}\left(
\sigma _{p}^{E}(\xi )\right) ,u_{g}^{p}\left( \sigma _{p}^{E}(\eta )\right)
}\right) \\
&=&\theta _{\left( \sigma _{pq}^{E}\circ u_{g}^{p}\right) \left( \sigma
_{p}^{E}(\xi )\right) ,\left( \sigma _{pq}^{E}\circ u_{g}^{p}\right) \left(
\sigma _{p}^{E}(\eta )\right) } \\
&=&\theta _{u_{g}^{q}\left( \sigma _{q}^{E}(\xi )\right) ,u_{g}^{q}\left(
\sigma _{q}^{E}(\eta )\right) } \\
&=&\beta _{g}^{u^{q}}\left( \theta _{\sigma _{q}^{E}(\xi ),\sigma
_{q}^{E}(\eta )}\right) \\
&=&\left( \beta _{g}^{u^{q}}\circ \left( \pi _{pq}^{A}\right) _{\ast
}\right) \left( \theta _{\sigma _{p}^{E}(\xi ),\sigma _{p}^{E}(\eta )}\right)
\end{eqnarray*}%
for all $\xi ,\eta \in E$. From these relations, we conclude that $(\alpha
_{g}^{u^{p}})_{p}$ and $\left( \beta _{g}^{u^{p}}\right) _{p}$ are inverse
systems of $C^{\ast }$-isomorphisms, and then $\lim\limits_{\underset{p\in
S(A)}{\leftarrow }}\alpha _{g}^{u^{p}}$ and $\lim\limits_{\underset{p\in S(A)%
}{\leftarrow }}\beta _{g}^{u^{p}}$ are isomorphisms of locally $C^{\ast }$%
-algebras. It is not difficult to check that $g\rightarrow \lim\limits_{%
\underset{p\in S(A)}{\leftarrow }}\alpha _{g}^{u^{p}}$ and $g\rightarrow
\lim\limits_{\underset{p\in S(A)}{\leftarrow }}\beta _{g}^{u^{p}}$ are
continuous inverse limit actions of $G$ on $A$ respectively $K(E)$.
Moreover, $\alpha _{g}^{u}=\lim\limits_{\underset{p\in S(A)}{\leftarrow }%
}\alpha _{g}^{u^{p}}$ of $G$ and $\beta _{g}^{u}=\lim\limits_{\underset{p\in
S(A)}{\leftarrow }}\beta _{g}^{u^{p}}$ for all $g\in G$.
\end{proof}

\section{Morita equivalence of group actions on locally $C^{*}$-algebras}

In this section we introduce the notion of (strong) Morita equivalence of
group actions on locally $C^{*}$-algebras and show that this is an
equivalence relation.

\begin{definition}
Let $G$ be a locally compact group, let $A$ and $B$ be two locally $C^{\ast
} $-algebras. Two actions $g\rightarrow \alpha _{g}$ and $g\rightarrow \beta
_{g}$ of $G$ on $A$ respectively $B$ are conjugate if there is an
isomorphism of locally $C^{\ast }$-algebras $\varphi :A\rightarrow B$ such
that $\alpha _{g}=\varphi ^{-1}\circ \beta _{g}\circ \varphi $ for each $%
g\in G.$
\end{definition}

\begin{remark}
Conjugacy of group actions on locally $C^{\ast }$-algebras is an equivalence
relation.
\end{remark}

\begin{lemma}
Let $g\rightarrow \alpha _{g}$ and $g\rightarrow \beta _{g}$ be two
conjugate actions of $G$ on the locally $C^{\ast }$-algebras $A$
respectively $B$. If $g\rightarrow \alpha _{g}$ is a continuous inverse
limit action of $G$ on $A$, then $g\rightarrow \beta _{g}$ is a continuous
inverse limit action of $G$ on $B.$
\end{lemma}

\begin{proof}
Suppose that $\alpha _{g}=\lim\limits_{\underset{p\in S(A)}{\leftarrow }%
}\alpha _{g}^{p}$ for each $g\in G$. Let $\varphi :A\rightarrow B$ be an
isomorphism of locally $C^{\ast }$-algebras such that $\alpha _{g}=\varphi
^{-1}\circ \beta _{g}\circ \varphi $ for each $g\in G$.

Let $p\in S(A)$. Then $p\circ \varphi ^{-1}$ is a continuous $C^{\ast }$%
-seminorm on $B.$ We show that $\{p\circ \varphi ^{-1}\}_{p\in S(A)}$ is a
cofinal subset of $S(B)$. For this, let $q\in S(B)$. Since $\varphi $ is an
isomorphism of locally $C^{\ast }$-algebras, there is $p\in S(A)$ such that 
\begin{equation*}
q\left( \varphi \left( a\right) \right) \leq p(a)
\end{equation*}%
for all $a\in A$. Then 
\begin{equation*}
q\left( b\right) =q\left( \varphi \left( \varphi ^{-1}(b)\right) \right)
\leq p\left( \varphi ^{-1}(b)\right) =\left( p\circ \varphi ^{-1}\right)
\left( b\right) 
\end{equation*}%
for all $b\in B$. Therefore $\{p\circ \varphi ^{-1}\}_{p\in S(A)}$ is a
cofinal subset of $S(B)$ and thus we can identify $B$ with $\lim\limits_{%
\underset{p\in S(A)}{\leftarrow }}B_{p\circ \varphi ^{-1}}$.

For each $p\in S(A)$, there is an isomorphism of $C^{\ast }$-algebras $%
\varphi _{p}:A_{p}\rightarrow B_{p\circ \varphi ^{-1}}$ such that 
\begin{equation*}
\varphi _{p}\circ \pi _{p}^{A}=\pi _{p\circ \varphi ^{-1}}^{B}\circ \varphi 
\end{equation*}%
since $\ker \pi _{p}^{A}=\ker \left( \pi _{p\circ \varphi ^{-1}}^{B}\circ
\varphi \right) $. The map $g\rightarrow \beta _{g}^{p}$, where $\beta
_{g}^{p}=\varphi _{p}\circ \alpha _{g}^{p}\circ \varphi _{p}^{-1}$, is a
continuous action of $G$ on $B_{p\circ \varphi ^{-1}}$ which is conjugate
with $\alpha ^{p}$. It is not difficult to check that $(\beta _{g}^{p})_{p}$
is an inverse system of $C^{\ast }$-isomorphisms and $\beta
_{g}=\lim\limits_{\underset{p\in S(A)}{\leftarrow }}\beta _{g}^{p}$ for each 
$g\in G$. Therefore $g\rightarrow \beta _{g}$ is a continuous inverse limit
action of $G$ on $B$.
\end{proof}

\begin{proposition}
Let $g\rightarrow \alpha _{g}$ and $g\rightarrow \beta _{g}$ be two
continuous inverse limit actions of $G$ on the locally $C^{\ast }$-algebras $%
A$ and $B$. If the actions $\alpha $ and $\beta $ are conjugate, then the
crossed products $G\times _{\alpha }A$ and $G\times _{\beta }B$ are
isomorphic.
\end{proposition}

\begin{proof}
By Lemma 4.3, we can suppose that $S(B)=\{p\circ \varphi ^{-1}\}_{p\in S(A)}$%
, where $\varphi $ is an isomorphism of locally $C^{\ast }$-algebras from $A$
to $B$ such that $\beta _{g}=\varphi \circ \alpha _{g}\circ \varphi ^{-1}.$
Then $B=\lim\limits_{\underset{p\in S(A)}{\leftarrow }}B_{p\circ \varphi
^{-1}},$ $\alpha _{g}=\lim\limits_{\underset{p\in S(A)}{\leftarrow }}\alpha
_{g}^{p}$ for each $g\in G$ and $\beta _{g}=\lim\limits_{\underset{p\in S(A)}%
{\leftarrow }}\beta _{g}^{p}$ for each $g\in G$. Moreover, for each $p\in
S(A)$, the actions $g\rightarrow \alpha _{g}^{p}$ of $G$ on $A_{p}$ and $%
g\rightarrow \beta _{g}^{p}$ of $G$ on $B_{p\circ \varphi ^{-1}}$ are
conjugate.

Let $p\in S(A),$ and let $\varphi _{p}$ be the isomorphism of $C^{\ast }$%
-algebras from $A_{p}$ to $B_{p\circ \varphi ^{-1}}$ such that $\beta
_{g}^{p}=\varphi _{p}\circ \alpha _{g}^{p}\circ \varphi _{p}^{-1}$. Then the
linear map $\phi _{p}:C_{c}(G,A_{p})\rightarrow C_{c}(G,B_{p\circ \varphi
^{-1}})$ defined by 
\begin{equation*}
\phi _{p}\left( f_{p}\right) =\varphi _{p}\circ f_{p}
\end{equation*}%
extends to an isomorphism of $C^{\ast }$-algebras from $G\times _{\alpha
^{p}}A_{p}$ to $G\times _{\beta ^{p}}B_{p\circ \varphi ^{-1}}$ \cite[pp. 66]%
{17}. Since 
\begin{eqnarray*}
\left( \pi _{pq}^{G\times _{\beta }B}\circ \phi _{p}\right) \left(
f_{p}\right)  &=&\pi _{pq}^{B}\circ \phi _{p}\left( f_{p}\right) =\pi
_{pq}^{B}\circ \varphi _{p}\circ f_{p} \\
&=&\left( \varphi _{q}\circ \pi _{pq}^{A}\right) \circ f_{p}=\varphi
_{q}\circ \left( \pi _{pq}^{A}\circ f_{p}\right)  \\
&=&\phi _{q}\left( \pi _{pq}^{G\times _{\alpha }A}\left( f_{p}\right)
\right) =\left( \phi _{q}\circ \pi _{pq}^{G\times _{\alpha }A}\right) \left(
f_{p}\right) 
\end{eqnarray*}%
for all $f_{p}\in $ $C_{c}(G,A_{p})$ and for all $p,q\in S(A)$ with $p\geq q,
$ $\left( \phi _{p}\right) _{p}$ is an inverse system of isomorphisms of $%
C^{\ast }$-algebras and then $\phi =$ $\lim\limits_{\underset{p\in S(A)}{%
\leftarrow }}\phi _{p}$ is an isomorphism of locally $C^{\ast }$-algebras
from $\lim\limits_{\underset{p\in S(A)}{\leftarrow }}G\times _{\alpha
^{p}}A_{p}$ to $\lim\limits_{\underset{p\in S(A)}{\leftarrow }}G\times
_{\beta ^{p}}B_{p}$. Therefore the locally $C^{\ast }$-algebras $G\times
_{\alpha }A$ and $G\times _{\beta }B$ are isomorphic.
\end{proof}

\begin{definition}
Let $G$ be a locally compact group, let $A$ and $B$ be two locally $C^{\ast
} $-algebras, and let $g\rightarrow \alpha _{g}$ and $g\rightarrow \beta
_{g} $ be two actions of $G$ on $A$ and $B$. We say that $\alpha $ and $%
\beta $ are ( strongly ) Morita equivalent, if there exist a full Hilbert $A$%
-module $E$, and a ( continuous ) action $g\rightarrow u_{g}$ of $G$ on $E$
such that the actions of $G$ on $A$ and $K(E)$ induced by $u$ are conjugate
with the actions $\alpha $ respectively $\beta $. We write $\alpha \thicksim
_{E,u}\beta $ if the actions $\alpha $ and $\beta $ are Morita equivalent
and $\alpha \thicksim _{E,u}^{s}\beta $ if the actions $\alpha $ and $\beta $
are strongly Morita equivalent.
\end{definition}

\begin{remark}
Two conjugate (continuous inverse limit ) actions of $G$ on the locally $%
C^{\ast }$-algebras $A$ and $B$ are (strongly) Morita equivalent.
\end{remark}

\begin{remark}
If $\alpha \thicksim _{E,u}^{s}\beta $ and $\alpha $ is a continuous inverse
limit action, then $u$ and $\beta $ are continuous inverse limit actions.
Indeed, for any $p\in S(A)$ we have%
\begin{equation*}
\overline{p}_{E}(u_{g}\left( \xi \right) )^{2}=p\left( \left\langle
u_{g}\left( \xi \right) ,u_{g}\left( \xi \right) \right\rangle \right)
=p\left( \alpha _{g}\left( \left\langle \xi ,\xi \right\rangle \right)
\right) =\overline{p}_{E}\left( \xi \right) ^{2}
\end{equation*}%
for all $g\in G$ and for all $\xi \in E$, and then, by Remark 3.6, $%
g\rightarrow u_{g}$ is a continuous inverse limit action of $G$ on $E$.
Moreover, by Proposition 3.8 and Lemma 4.3, $g\rightarrow \beta _{g}$ is a
continuous inverse limit action.
\end{remark}

Let $E$ be a full Hilbert $A$-module and let $\widetilde{E}=K(E,A)$. Then $%
\widetilde{E}$ is a Hilbert $K(E)$-module in a natural way \cite[pp. 805-806]%
{6}. Moreover, $\widetilde{E}$ is full and the locally $C^{\ast }$-algebras $%
K(\widetilde{E})$ and $A$ are isomorphic \cite[pp. 805-806]{6}.

\begin{proposition}
Let $G$ be a locally compact group and let $g\rightarrow u_{g}$ be an action
of $G$ on a full Hilbert $A$-module $E$. Then $u$ induces an action $%
g\rightarrow \widetilde{u}_{g}$ of $G$ on the Hilbert $K(E)$-module $%
\widetilde{E}$ such that $\widetilde{u}_{g}\left( \theta _{a,\xi }\right)
=\theta _{\alpha _{g}^{u}\left( a\right) ,u_{g}\left( \xi \right) }$ for all 
$a\in A$ and for all $\xi \in E$. Moreover, if $u$ is a continuous inverse
limit action, then $\widetilde{u}$ is a continuous inverse limit action.
\end{proposition}

\begin{proof}
Let $g\in G$. Consider the linear map $\widetilde{u}_{g}:\Theta \left(
E,A\right) \rightarrow \Theta \left( E,A\right) $ defined by 
\begin{equation*}
\widetilde{u}_{g}(\theta _{a,\xi })=\theta _{\alpha _{g}^{u}\left( a\right)
,u_{g}\left( \xi \right) }.
\end{equation*}%
Since 
\begin{eqnarray*}
\overline{p}_{\widetilde{E}}\left( \widetilde{u}_{g}(\theta _{a,\xi
})\right)  &=&\widetilde{p}_{L(E,A)}\left( \widetilde{u}_{g}(\theta _{a,\xi
})\right) =\widetilde{p}_{L(E,A)}\left( \theta _{\alpha _{g}^{u}\left(
a\right) ,u_{g}\left( \xi \right) }\right)  \\
&=&\sup \{p\left( \alpha _{g}^{u}\left( a\right) \left\langle u_{g}\left(
\xi \right) ,u_{g}\left( \zeta \right) \right\rangle \right) ;\overline{p}%
_{E}(u_{g}\left( \zeta \right) )\leq 1\} \\
&=&\widetilde{r}_{L(E,A)}\left( \theta _{a,\xi }\right) =\overline{r}_{%
\widetilde{E}}\left( \theta _{a,\xi }\right) 
\end{eqnarray*}%
where $r=p\circ \alpha _{g}^{u}\in S(A)$, for all $a\in A$ and for all $\xi
\in E$, $\widetilde{u}_{g}$ extends to a continuous linear map, denoted also
by $\widetilde{u}_{g}$, from $K(E,A)$ to $K(E,A)$. From 
\begin{eqnarray*}
\left\langle \widetilde{u}_{g}(\theta _{a,\xi }),\widetilde{u}_{g}(\theta
_{b,\eta })\right\rangle  &=&\left\langle \theta _{\alpha _{g}^{u}\left(
a\right) ,u_{g}\left( \xi \right) },\theta _{\alpha _{g}^{u}\left( b\right)
,u_{g}\left( \eta \right) }\right\rangle  \\
&=&\theta _{u_{g}\left( \xi \right) ,\alpha _{g}^{u}\left( a\right) }\circ
\theta _{\alpha _{g}^{u}\left( b\right) ,u_{g}\left( \eta \right) } \\
&=&\theta _{u_{g}\left( \xi a^{\ast }\right) ,u_{g}\left( \eta b^{\ast
}\right) }=\beta _{g}^{u}\left( \theta _{\xi a^{\ast }\eta b^{\ast }}\right) 
\\
&=&\beta _{g}^{u}\left( \left\langle \theta _{a,\xi },\theta _{b,\eta
}\right\rangle \right) 
\end{eqnarray*}%
for all $\xi ,\eta \in E$ and for all $a,b\in A$, and taking into account
that $\widetilde{u}_{g}$ is continuous and $\Theta \left( E,A\right) $ is
dense in $K(E,A)$, we conclude that $\widetilde{u}_{g}$ is a morphism of
Hilbert modules. Moreover, since $\widetilde{u}_{g}$ is invertible and $%
\left( \widetilde{u}_{g}\right) ^{-1}=\widetilde{u}_{g^{-1}},$ $\widetilde{u}%
_{g}$ is an isomorphism of Hilbert modules. A simple calculation shows that
the map $g\rightarrow \widetilde{u}_{g}$ is an action of $G$ on $\widetilde{E%
}$.

Now, we suppose that $g\rightarrow u_{g}$ is a continuous inverse limit
action. Then $u_{g}=\lim\limits_{\underset{p\in S(A)}{\leftarrow }}u_{g}^{p}$
for each $g\in G$ and $g\rightarrow u_{g}^{p}$ is a continuous action of $G$
on $E_{p}$ for each $p\in S(A)$. For each $p\in S(A),$ the continuous action 
$u^{p}$ of $G$ on $E_{p}$ induces a continuous action $\widetilde{u^{p}}$ of 
$G$ on $\widetilde{E_{p}}$. It is not difficult to check that for each $g\in
G,$ $\left( \widetilde{u}_{g}^{p}\right) _{p}$ is an inverse system of
isomorphisms of Hilbert modules and then $g\rightarrow \lim\limits_{\underset%
{p\in S(A)}{\leftarrow }}\widetilde{u}_{g}^{p}$ is a continuous inverse
limit action of $G$ on $\lim\limits_{\underset{p\in S(A)}{\leftarrow }}$ $%
\widetilde{E_{p}}$. Since the Hilbert $A$-modules $\widetilde{E}$ and $%
\lim\limits_{\underset{p\in S(A)}{\leftarrow }}$ $\widetilde{E_{p}}$ are
unitarily equivalent \cite[pp. 805-806]{6}, the action $g\rightarrow
\lim\limits_{\underset{p\in S(A)}{\leftarrow }}\widetilde{u}_{g}^{p}$ of $G$
on $\lim\limits_{\underset{p\in S(A)}{\leftarrow }}$ $\widetilde{E_{p}}$ can
be identified with an action of $G$ on $\widetilde{E}$. Moreover,%
\begin{equation*}
\left( \lim\limits_{\underset{p\in S(A)}{\leftarrow }}\widetilde{u}%
_{g}^{p}\right) \left( \theta _{a,\xi })\right) =\theta _{\alpha
_{g}^{u}\left( a\right) ,u_{g}\left( \xi \right) }
\end{equation*}
for all $a\in A$ and for all $\xi \in E$.
\end{proof}

\begin{remark}
Let $G$ be a locally compact group and let $g\rightarrow u_{g}$ be an action
of $G$ on a full Hilbert $A$-module $E$.

\begin{enumerate}
\item The action of $G$ on $K(E)$ induced by $\widetilde{u}$ coincides with
the action of $G$ on $K(E)$ induced by $u$.

\item Since $\varphi :K(\widetilde{E})\rightarrow A$ defined by 
\begin{equation*}
\varphi \left( \theta _{\theta _{a,\xi },\theta _{b,\eta }}\right)
=\left\langle \xi a^{\ast },\eta b^{\ast }\right\rangle
\end{equation*}%
is an isomorphism of locally $C^{\ast }$-algebras, and since 
\begin{eqnarray*}
\left( \alpha _{g}^{u}\circ \varphi \right) \left( \theta _{\theta _{a,\xi
},\theta _{b,\eta }}\right) &=&\alpha _{g}^{u}\left( \left\langle \xi
a^{\ast },\eta b^{\ast }\right\rangle \right) \\
&=&\left\langle u_{g}\left( \xi \right) \alpha _{g}^{u}\left( a\right)
^{\ast },u_{g}\left( \eta \right) \alpha _{g}^{u}\left( b\right) ^{\ast
}\right\rangle \\
&=&\varphi \left( \theta _{\theta \alpha _{g}^{u}\left( a\right)
,u_{g}\left( \xi \right) ,\theta \alpha _{g}^{u}\left( b\right) ,u_{g}\left(
\eta \right) }\right) \\
&=&\varphi \left( \theta _{\widetilde{u}_{g}\left( \theta _{a,\xi }\right) ,%
\widetilde{u}_{g}\left( \theta _{b,\xi }\right) }\right) \\
&=&\left( \varphi \circ \beta _{g}^{\widetilde{u}}\right) \left( \theta
_{\theta _{a,\xi },\theta _{b,\eta }}\right)
\end{eqnarray*}%
for all $a,b\in A,$ for all $\xi ,\eta \in E$ and for all $g\in G,$ the
action of $G$ on $A$ induced by $u$ is conjugate with the action of $G$ on $%
K(\widetilde{E})$ induced by $\widetilde{u}.$
\end{enumerate}
\end{remark}

Let $E$ be a Hilbert $A$-module, let $F$ be a Hilbert $B$-module and let $%
\Phi :A\rightarrow L(F)$ be a non-degenerate representation of $A$ on $F$.
The tensor product $E\otimes _{A}F$ of $E$ and $F$ over $A$ becomes a
pre-Hilbert $B$-module with the action of $B$ on $E\otimes _{A}F$ defined by 
\begin{equation*}
\left( \xi \otimes _{A}\eta \right) b=\xi \otimes _{A}\eta b
\end{equation*}%
and the inner product defined by 
\begin{equation*}
\left\langle \xi _{1}\otimes _{A}\eta _{1},\xi _{2}\otimes _{A}\eta
_{2}\right\rangle _{\Phi }=\left\langle \eta _{1},\Phi \left( \left\langle
\xi _{1},\xi _{2}\right\rangle \right) \eta _{2}\right\rangle .
\end{equation*}%
The completion of $E\otimes _{A}F$ with respect to the topology determined
by the inner product is called the interior tensor product of the Hilbert
modules $E$ and $F$ using $\Phi $ and it is denoted by $E\otimes _{\Phi }F$ 
\cite[Proposition 4.1]{11}. An element in $E\otimes _{\Phi }F$ is denoted by 
$\xi \otimes _{\Phi }\eta .$

\begin{definition}
Let $g\rightarrow \alpha _{g}$ be an action of $G$ on $A$ and let $\Phi
:A\rightarrow L(F)$ be a non-degenerate representation of $A$ on a Hilbert $%
B $-module $F$. An action $g\rightarrow v_{g}$ of $G$ on $F$ is $\Phi $
covariant relative to $\alpha $ if 
\begin{equation*}
v_{g}\Phi \left( a\right) v_{g^{-1}}=\Phi \left( \alpha _{g}(a)\right)
\end{equation*}%
for all $a\in A$ and for all $g\in G$.
\end{definition}

\begin{proposition}
Let $G$ be a locally compact group, let $E$ be a full Hilbert $A$-module,
let $F$ be a full Hilbert $B$-module and let $\Phi :A\rightarrow L(F)$ be a
non-degenerate representation of $A$ on $F$. If $g\rightarrow u_{g}$ is an
action of $G$ on $E$ and $g\rightarrow v_{g}$ is an action of $G$ on $F$
which is $\Phi $ covariant relative to $\alpha ^{u}$ ( the action of $G$ on $%
A$ induced by $u$ ), then there is a unique action $g\rightarrow w_{g}^{u,v} 
$ of $G$ on $E\otimes _{\Phi }F$ such that 
\begin{equation*}
w_{g}^{u,v}\left( \xi \otimes _{\Phi }\eta \right) =u_{g}\left( \xi \right)
\otimes _{\Phi }v_{g}\left( \eta \right)
\end{equation*}%
for all $\xi \in E$, for all $\eta \in F$ and for all $g\in G$. Moreover, if 
$g$ $\rightarrow u_{g}$ is a continuous action of $G$ on $E$ and $%
g\rightarrow v_{g}$ is a continuous inverse limit action of $G$ on $F$, then 
$g\rightarrow w_{g}^{u,v}$ is a continuous inverse limit action of $G$ on $%
E\otimes _{\Phi }F$.
\end{proposition}

\begin{proof}
Let $g\in G$. Since 
\begin{equation*}
u_{g}\left( \xi a\right) \otimes _{\Phi }v_{g}\left( \eta \right)
=u_{g}\left( \xi \right) \otimes _{\Phi }v_{g}\Phi \left( a\right) \left(
\eta \right)
\end{equation*}%
for all $\xi \in E$, for all $\eta \in F$ and for all $a\in A$, and since 
\begin{eqnarray*}
\overline{q}_{E\otimes _{\Phi }F}\left( u_{g}\left( \xi \right) \otimes
_{\Phi }v_{g}\left( \eta \right) \right) ^{2} &=&q\left( \left\langle
v_{g}\left( \eta \right) ,\Phi \left( \alpha _{g}^{u}\left( \left\langle \xi
,\xi \right\rangle \right) \right) v_{g}\left( \eta \right) \right\rangle
\right) \\
&=&q\left( \left\langle v_{g}\left( \eta \right) ,v_{g}\left( \Phi \left(
\left\langle \xi ,\xi \right\rangle \right) \left( \eta \right) \right)
\right\rangle \right) \\
&=&q\left( \alpha _{g}^{v}\left( \left\langle \eta ,\Phi \left( \left\langle
\xi ,\xi \right\rangle \right) \left( \eta \right) \right\rangle \right)
\right) \\
&=&\left( q\circ \alpha _{g}^{v}\right) \left( \left\langle \xi \otimes
_{\Phi }\eta ,\xi \otimes _{\Phi }\eta \right\rangle _{\Phi }\right) \\
&=&\overline{r}_{E\otimes _{\Phi }F}\left( \xi \otimes _{\Phi }\eta \right)
^{2}
\end{eqnarray*}%
where $r=q\circ \alpha _{g}^{v}\in S(B)$, for all $\xi \in E$ and for all $%
\eta \in F$, there is a continuous linear map $w_{g}^{u,v}:E\otimes _{\Phi
}F\rightarrow E\otimes _{\Phi }F$ such that 
\begin{equation*}
w_{g}^{u,v}\left( \xi \otimes _{\Phi }\eta \right) =u_{g}\left( \xi \right)
\otimes _{\Phi }v_{g}\left( \eta \right) .
\end{equation*}%
From 
\begin{equation*}
\left\langle w_{g}^{u,v}\left( \xi _{1}\otimes _{\Phi }\eta _{1}\right)
,w_{g}^{u,v}\left( \xi _{2}\otimes _{\Phi }\eta _{2}\right) \right\rangle
_{\Phi }=\alpha _{g}^{v}\left( \left\langle \xi _{1}\otimes _{\Phi }\eta
_{1},\xi _{2}\otimes _{\Phi }\eta _{2}\right\rangle _{\Phi }\right)
\end{equation*}%
for all $\xi _{1},\xi _{2}\in E$ and for all $\eta _{1},\eta _{2}\in F$ and
taking into account that $w_{g}^{u,v}\circ
w_{g^{-1}}^{u,v}=w_{g-1}^{u,v}\circ w_{g}^{u,v}=$id $_{E\otimes _{\Phi }F},$
we conclude that $w_{g}^{u,v}\in $Aut$\left( E\otimes _{\Phi }F\right) $.
Since $\{\xi \otimes _{\Phi }\eta ;\xi \in E,\eta \in F\}$ generates $%
E\otimes _{\Phi }F$, $w_{g}^{u,v}$ is the unique linear map from $E\otimes
_{\Phi }F$ to $E\otimes _{\Phi }F$ such that 
\begin{equation*}
w_{g}^{u,v}\left( \xi \otimes _{\Phi }\eta \right) =u_{g}\left( \xi \right)
\otimes _{\Phi }v_{g}\left( \eta \right)
\end{equation*}%
for all $\xi \in E$, for all $\eta \in F$ and for all $g\in G$. It is not
difficult to check that $g\rightarrow w_{g}^{u,v}$ is an action of $G$ on $%
E\otimes _{\Phi }F$.

To show that the action $w^{u,v}$ of $G$ on $E\otimes _{\Phi }F$ is a
continuous inverse limit action if $u$ is continuous and $v$ is a continuous
inverse limit action, we partition the proof into two steps.

\textit{Step 1}. We suppose that $B$ is a $C^{\ast }$-algebra.

Let $x_{0}\in E\otimes _{\Phi }F$ and $\varepsilon >0.$ Then there is $\xi
_{0}\in E$ and $\eta _{0}\in F$ such that%
\begin{equation*}
\left\Vert x_{0}-\xi _{0}\otimes _{\Phi }\eta _{0}\right\Vert \leq
\varepsilon /6.
\end{equation*}%
Since $\Phi $ is a representation of $A$ on $F$ there is $p\in S(A)$ such
that 
\begin{equation*}
\left\Vert \Phi \left( a\right) \right\Vert \leq p(a)
\end{equation*}%
for all $a\in A$, and since $g\rightarrow u_{g}$ and $g\rightarrow v_{g}$
are continuous, there is a neighborhood $U_{0}$ of $g_{0}$ such that 
\begin{equation*}
\overline{p}_{E}\left( u_{g}\left( \xi _{0}\right) -\xi _{0}\right) \leq
\min \{\sqrt{\varepsilon /6},\varepsilon /\left( 6\left\Vert \eta
_{0}\right\Vert \right) \}
\end{equation*}%
and 
\begin{equation*}
\left\Vert v_{g}\left( \eta _{0}\right) -\eta _{0}\right\Vert \leq \min \{%
\sqrt{\varepsilon /6},\varepsilon /\left( 6\overline{p}_{E}(\xi _{0})\right)
\}
\end{equation*}%
for all $g\in U_{0}$. Then, since 
\begin{eqnarray*}
\left\Vert \xi \otimes _{\Phi }\eta \right\Vert ^{2} &=&\left\Vert
\left\langle \xi \otimes _{\Phi }\eta ,\xi \otimes _{\Phi }\eta
\right\rangle _{\Phi }\right\Vert =\left\Vert \left\langle \eta ,\Phi \left(
\left\langle \xi ,\xi \right\rangle \right) \eta \right\rangle \right\Vert \\
&\leq &\left\Vert \eta \right\Vert \left\Vert \Phi \left( \left\langle \xi
,\xi \right\rangle \right) \eta \right\Vert \leq \left\Vert \eta \right\Vert
^{2}\left\Vert \Phi \left( \left\langle \xi ,\xi \right\rangle \right)
\right\Vert \\
&\leq &\left\Vert \eta \right\Vert ^{2}p\left( \left\langle \xi ,\xi
\right\rangle \right) =\left\Vert \eta \right\Vert ^{2}\overline{p}%
_{E}\left( \xi \right) ^{2}
\end{eqnarray*}%
for all $\xi \in E$ and for all $\eta \in F$, we have 
\begin{eqnarray*}
\left\Vert w_{g}^{u,v}\left( x_{0}\right) -x_{0}\right\Vert &\leq
&\left\Vert w_{g}^{u,v}\left( x_{0}-\xi _{0}\otimes _{\Phi }\eta _{0}\right)
\right\Vert +\left\Vert w_{g}^{u,v}\left( \xi _{0}\otimes _{\Phi }\eta
_{0}\right) -\xi _{0}\otimes _{\Phi }\eta _{0}\right\Vert \\
&&+\left\Vert x_{0}-\xi _{0}\otimes _{\Phi }\eta _{0}\right\Vert \\
&\leq &2\left\Vert x_{0}-\xi _{0}\otimes _{\Phi }\eta _{0}\right\Vert
+\left\Vert u_{g}\left( \xi _{0}\right) \otimes _{\Phi }v_{g}\left( \eta
_{0}\right) -\xi _{0}\otimes _{\Phi }\eta _{0}\right\Vert \\
&\leq &\varepsilon /3+\left\Vert \left( u_{g}\left( \xi _{0}\right) -\xi
_{0}\right) \otimes _{\Phi }\left( v_{g}\left( \eta _{0}\right) -\eta
_{0}\right) \right\Vert \\
&&+\left\Vert \xi _{0}\otimes _{\Phi }\left( v_{g}\left( \eta _{0}\right)
-\eta _{0}\right) \right\Vert +\left\Vert \left( u_{g}\left( \xi _{0}\right)
-\xi _{0}\right) \otimes _{\Phi }\eta _{0}\right\Vert \\
&\leq &\varepsilon /3+\left\Vert v_{g}\left( \eta _{0}\right) -\eta
_{0}\right\Vert \overline{p}_{E}\left( u_{g}\left( \xi _{0}\right) -\xi
_{0}\right) \\
&&+\left\Vert v_{g}\left( \eta _{0}\right) -\eta _{0}\right\Vert \overline{p}%
_{E}\left( \xi _{0}\right) +\left\Vert \eta _{0}\right\Vert \overline{p}%
_{E}\left( u_{g}\left( \xi _{0}\right) -\xi _{0}\right) \\
&\leq &\varepsilon
\end{eqnarray*}%
for all $g\in U_{0}$. Therefore the action $w^{u,v}$ of $G$ on $E\otimes
_{\Phi }F$ is continuous.

\textit{Step 2.} The general case when $B$ is a locally $C^{\ast }$-algebra.
By \cite[Proposition 4.2]{11}, $E\otimes _{\Phi }F$ can be identified with $%
\lim\limits_{\underset{q\in S(B)}{\leftarrow }}E\otimes _{\Phi _{q}}F_{q}$,
where $\Phi _{q}=\left( \pi _{q}^{B}\right) _{\ast }\circ \Phi $ for each $%
q\in S(B)$.

Let $q\in S(B)$. From 
\begin{eqnarray*}
v_{g}^{q}\Phi _{q}\left( a\right) v_{g^{-1}}^{q}\left( \sigma _{q}^{F}\left(
\eta \right) \right) &=&\left( v_{g}^{q}\Phi _{q}\left( a\right) \right)
\left( \sigma _{q}^{F}\left( v_{g^{-1}}\left( \eta \right) \right) \right) \\
&=&v_{g}^{q}\left( \sigma _{q}^{F}\left( \Phi \left( a\right)
v_{g^{-1}}\left( \eta \right) \right) \right) \\
&=&\sigma _{q}^{F}\left( v_{g}^{q}\Phi \left( a\right) v_{g^{-1}}\left( \eta
\right) \right) =\sigma _{q}^{F}\left( \Phi \left( \alpha _{g}^{u}\left(
a\right) \right) \left( \eta \right) \right) \\
&=&\Phi _{q}\left( \alpha _{g}^{u}\left( a\right) \right) \left( \sigma
_{q}^{F}\left( \eta \right) \right)
\end{eqnarray*}%
for all $\eta \in F$, for all $a\in A$ and for all $g\in G$, we conclude
that the action $g\rightarrow v_{g}^{q}$ of $G$ on $F_{q}$ is $\Phi _{q}$
covariant relative to $\alpha ^{u}$. Then, by the first step of this prof, $%
g\rightarrow w_{g}^{u,v^{q}}$ is a continuous action of $G$ on $E\otimes
_{\Phi _{q}}F_{q}$. It is not difficult to check that for each $g\in G$, $%
\left( w_{g}^{u,v^{q}}\right) _{q}$ is an inverse system of isomorphisms of
Hilbert $C^{\ast }$-modules, and the map $g\rightarrow \lim\limits_{\underset%
{q\in S(B)}{\leftarrow }}w_{g}^{u,v^{q}}$ is a continuous inverse limit
action of $G$ on $\lim\limits_{\underset{q\in S(B)}{\leftarrow }}E\otimes
_{\Phi _{q}}F_{q}$. Moreover, this action can be identified with the action $%
g\rightarrow w_{g}^{u,v}$ of $G$ on $E\otimes _{\Phi }F.$
\end{proof}

\begin{remark}
\begin{enumerate}
\item The action of $G$ on $B$ induced by the action $w^{u,v}$ of $G$ on $%
E\otimes _{\Phi }F$ coincides with the action of $G$ on $B$ induced by $v$.

\item Suppose that $\Phi :A\rightarrow K(E)$ is an isomorphism of locally $%
C^{\ast }$-algebras. Then $\Phi _{\ast }:K(E)\rightarrow K(E\otimes _{\Phi
}F)$ defined by $\Phi _{\ast }\left( T\right) \left( \xi \otimes _{\Phi
}\eta \right) =T(\xi )\otimes _{\Phi }\eta $ is an isomorphism of locally $%
C^{\ast }$-algebras \cite[Proposition 4.4 and Corollary 4.6]{11}. In this
case, the action of $G$ on $K(E\otimes _{\Phi }F)$ induced by $w^{u,v}$ is
conjugate with the action of $G$ on $K(E)$ induced by $u$. Indeed, we have 
\begin{eqnarray*}
\left( \beta _{g}^{w^{u,v}}\circ \Phi _{\ast }\right) \left( T\right) \left(
\xi \otimes _{\Phi }\eta \right) &=&\left( w_{g}^{u,w}\Phi _{\ast }\left(
T\right) w_{g^{-1}}^{u,w}\right) \left( \xi \otimes _{\Phi }\eta \right) \\
&=&\left( w_{g}^{u,w}\Phi _{\ast }\left( T\right) \right) \left(
u_{g^{-1}}\left( \xi \right) \otimes _{\Phi }v_{g^{-1}}\left( \eta \right)
\right) \\
&=&w_{g}^{u,w}\left( T\left( u_{g^{-1}}\left( \xi \right) \right) \otimes
_{\Phi }v_{g^{-1}}\left( \eta \right) \right) \\
&=&u_{g}Tu_{g^{-1}}\left( \xi \right) \otimes _{\Phi }\eta \\
&=&\beta _{g}^{u}\left( T\right) \left( \xi \right) \otimes _{\Phi }\eta \\
&=&\Phi _{\ast }\left( \beta _{g}^{u}\left( T\right) \right) \left( \xi
\otimes _{\Phi }\eta \right) \\
&=&\left( \Phi _{\ast }\circ \beta _{g}^{u}\right) \left( T\right) \left(
\xi \otimes _{\Phi }\eta \right)
\end{eqnarray*}%
for all $\xi \in E$ and for all $\eta \in F$.
\end{enumerate}
\end{remark}

\begin{proposition}
Morita equivalence of group actions on locally $C^{\ast }$-algebras is an
equivalence relation.
\end{proposition}

\begin{proof}
Let $\alpha $ be an action of $G$ on the locally $C^{\ast }$-algebra $A$.
Clearly, we can regard $\alpha $ as an action of $G$ on the Hilbert $A$%
-module $A$. Since the Hilbert $A$-module $A$ is full and since $K(A)$ is
isomorphic with $A,$ $\alpha \backsim _{A,\alpha }\alpha $. Therefore the
relation is reflexive.

From Proposition 4.8 and Remark 4.9, we conclude that the relation is
symmetric.

To show that the relation is transitive, let $\alpha ,\beta ,\gamma $ be
three actions of $G$ on the locally $C^{\ast }$-algebras $A,$ $B$ and $C$
such that $\alpha \backsim _{E,u}\beta $ and $\beta \backsim _{F,v}\gamma $.
By \cite[the proof of Proposition 4.4]{6}, $F\otimes _{i}E$, where $i$ is
the embedding of $K(E)$ in $L(E)$, is a full Hilbert $A$-module such that
the locally $C^{\ast }$-algebras $K(E)$ and $K(F\otimes _{i}E)$ are
isomorphic. Since%
\begin{equation*}
i\left( \beta _{g}^{u}\left( \theta _{\xi ,\eta }\right) \right)
=u_{g}i\left( \theta _{\xi ,\eta }\right) u_{g^{-1}}
\end{equation*}%
for all $g\in G$ and for all $\xi ,\eta \in E$, the action $u$ is covariant
with respect to $i$ relative to $\beta ^{u}.$ Then, from Proposition 4.11
and Remark 4.12, we conclude that the pair $\left( F\otimes
_{i}E,w^{u,v}\right) $ implements a Morita equivalence between $\alpha $ and 
$\gamma $.
\end{proof}

In the same way as in the proof of Proposition 4.13 and using Remark 4.7, we
obtain the following proposition.

\begin{proposition}
Strong Morita equivalence of continuous inverse limit group actions on
locally $C^{\ast }$-algebras is an equivalence relation.
\end{proposition}

\section{Crossed products of locally $C^{\ast }$-algebras associated with
strong Morita equivalent actions}

Let $G$ be a locally compact group and let $E$ be a Hilbert module over a
locally $C^{*}$-algebra $A.$

Let $C_{c}(G,E)$ be the vector space of all continuous functions from $G$ to 
$E$ with compact support.

\begin{remark}
If $\widehat{\xi },\widehat{\eta }\in C_{c}(G,E),$ then the function $%
\left\langle \widehat{\xi },\widehat{\eta }\right\rangle :G\rightarrow A$
defined by $\left\langle \widehat{\xi },\widehat{\eta }\right\rangle \left(
s\right) =\left\langle \widehat{\xi }(s),\widehat{\eta }(s)\right\rangle $
is with compact support.
\end{remark}

\begin{lemma}
Let $\widehat{\xi }\in C_{c}(G,E).$ Then there is a unique element $%
\int\limits_{G}\widehat{\xi }(s)ds\in E,$ such that 
\begin{equation*}
\left\langle \int\limits_{G}\widehat{\xi }(s)ds,\eta \right\rangle
=\int\limits_{G}\left\langle \widehat{\xi }(s),\eta \right\rangle ds
\end{equation*}%
for all $\eta \in E.$ Moreover,

(1) $\overline{p}_{E}\left( \int\limits_{G}\widehat{\xi }(s)ds\right) \leq
M_{\widehat{\xi }}\sup \{\overline{p}_{E}\left( \widehat{\xi }(s)\right)
;s\in G\}$ for some positive number $M_{\widehat{\xi }};$

(2) $\left( \int\limits_{G}\widehat{\xi }(s)ds\right) a=\int\limits_{G}%
\widehat{\xi }(s)ads$ for all $a\in A;$

(3) $T\left( \int\limits_{G}\widehat{\xi }(s)ds\right)
=\int\limits_{G}T\left( \widehat{\xi }(s)\right) ds$ for all $T\in L(E,F).$
\end{lemma}

\begin{proof}
Consider the map $T\left( \widehat{\xi }\right) :E\rightarrow A$ defined by 
\begin{equation*}
T\left( \widehat{\xi }\right) \left( \eta \right)
=\int\limits_{G}\left\langle \widehat{\xi }\left( s\right) ,\eta
\right\rangle ds.
\end{equation*}%
By Lemma 2.2, $T\left( \widehat{\xi }\right) $ is a module morphism. Let $%
p\in S(A)$. Then 
\begin{eqnarray*}
\widetilde{p}_{L(E,A)}\left( T\left( \widehat{\xi }\right) \right)  &=&\sup
\{p\left( T\left( \widehat{\xi }\right) \left( \eta \right) \right) ;%
\overline{p}_{E}(\eta )\leq 1\} \\
&=&\sup \{p\left( \int\limits_{G}\left\langle \widehat{\xi }\left( s\right)
,\eta \right\rangle ds\right) ;\overline{p}_{E}(\eta )\leq 1\} \\
&\leq &M_{\widehat{\xi }}\sup \{\sup \{p\left( \left\langle \widehat{\xi }%
\left( s\right) ,\eta \right\rangle \right) ;s\in G\};\overline{p}_{E}(\eta
)\leq 1\} \\
&\leq &M_{\widehat{\xi }}\sup \{\overline{p}_{E}\left( \widehat{\xi }\left(
s\right) \right) ;s\in G\}
\end{eqnarray*}%
where $M_{\widehat{\xi }}=\int_{G_{\widehat{\xi }}}ds,$ $G_{\widehat{\xi }}=$
supp$\widehat{\xi }$.

If $\widehat{\xi }=\xi \otimes f,$ $\xi \in E$ and $f\in C_{c}(G),$ it is
not difficult to check that $T\left( \widehat{\xi }\right) =T_{\zeta },$
where $\zeta =\left( \int_{G}f(s)ds\right) \xi $ and $T_{\zeta }\left( \eta
\right) =\left\langle \zeta ,\eta \right\rangle $ for all $\eta \in E.$
Therefore $T\left( \widehat{\xi }\right) \in K(E,A)$.

Now suppose that $\widehat{\xi }\in C_{c}(G,E)$. For $\varepsilon >0$ and $%
p\in S(A)$, there exist $\xi _{i}\in E$ and $f_{i}\in C_{c}(G),$ $%
i=1,2,...,n,$ such that 
\begin{equation*}
\sup \{\overline{p}_{E}\left( \widehat{\xi }\left( s\right)
-\sum\limits_{i=1}^{n}f_{i}(s)\xi _{i}\right) ,s\in G\}\leq \varepsilon /M_{%
\widehat{\xi }}.
\end{equation*}%
Then 
\begin{eqnarray*}
\widetilde{p}_{L(E,A)}\left( T\left( \widehat{\xi }\right) -T\left(
\sum\limits_{i=1}^{n}\xi _{i}\otimes f_{i}\right) \right) &\leq &M_{\widehat{%
\xi }}\sup \{\overline{p}_{E}\left( \widehat{\xi }\left( s\right)
-\sum\limits_{i=1}^{n}f_{i}(s)\xi _{i}\right) ,s\in G\} \\
&\leq &\varepsilon .
\end{eqnarray*}%
From these facts we conclude that $T\left( \widehat{\xi }\right) \in K(E,A)$%
. Therefore there is a unique element $\int\limits_{G}\widehat{\xi }(s)ds\in
E$ such that 
\begin{equation*}
T\left( \widehat{\xi }\right) \left( \eta \right) =\left\langle
\int\limits_{G}\widehat{\xi }(s)ds,\eta \right\rangle
\end{equation*}%
and so 
\begin{equation*}
\left\langle \int\limits_{G}\widehat{\xi }(s)ds,\eta \right\rangle
=\int\limits_{G}\left\langle \widehat{\xi }(s),\eta \right\rangle ds
\end{equation*}%
for all $\eta \in E$. Moreover, $\widetilde{p}_{L(E,A)}\left( T\left( 
\widehat{\xi }\right) \right) =\overline{p}_{E}\left( \int\limits_{G}%
\widehat{\xi }(s)ds\right) $ \cite[Lemma 2.1.3 and Corollary 1.2.3 ]{10},
and thus the relation (1) is verified.

Let $a\in A$. Then 
\begin{eqnarray*}
\left\langle \int\limits_{G}\widehat{\xi }(s)ads,\eta \right\rangle
&=&\int\limits_{G}\left\langle \widehat{\xi }(s)a,\eta \right\rangle
ds=a^{\ast }\int\limits_{G}\left\langle \widehat{\xi }(s),\eta \right\rangle
ds \\
&=&a^{\ast }\left\langle \int\limits_{G}\widehat{\xi }(s)ds,\eta
\right\rangle =\left\langle \left( \int\limits_{G}\widehat{\xi }(s)ds\right)
a,\eta \right\rangle
\end{eqnarray*}%
for all $\eta \in E$. This implies that 
\begin{equation*}
\int\limits_{G}\widehat{\xi }(s)ads=\left( \int\limits_{G}\widehat{\xi }%
(s)ds\right) a
\end{equation*}%
and thus the relation (2) is verified. Let $T\in L(E,F)$. From 
\begin{eqnarray*}
\left\langle T\left( \int\limits_{G}\widehat{\xi }(s)ds\right) ,\eta
\right\rangle &=&\left\langle \int\limits_{G}\widehat{\xi }(s)ds,T^{\ast
}(\eta )\right\rangle =\int\limits_{G}\left\langle \widehat{\xi }(s),T^{\ast
}(\eta )\right\rangle ds \\
&=&\int\limits_{G}\left\langle T\left( \widehat{\xi }(s)\right) ,\eta
)\right\rangle ds=\left\langle \int\limits_{G}T\left( \widehat{\xi }%
(s)\right) ds,\eta )\right\rangle
\end{eqnarray*}%
we conclude that $T\left( \int\limits_{G}\widehat{\xi }(s)ds\right)
=\int\limits_{G}T\left( \widehat{\xi }(s)\right) ds,$ and then the relation
(3) is verified too.
\end{proof}

Let $g\rightarrow u_{g}$ be a continuous inverse limit action of $G$ on a
full Hilbert $A$-module $E$. We can suppose that $u_{g}=\lim\limits_{%
\underset{p\in S(A)}{\leftarrow }}u_{g}^{p}$ for each $g\in G$, where $%
g\rightarrow u_{g}^{p}$, $p\in S(A)$ are continuous actions of $G$ on $E_{p}$%
.

Let $\widehat{\xi }\in C_{c}(G,E)$ and $f\in C_{c}(G,A)$. It is not
difficult to check that the function 
\begin{equation*}
\left( s,t\right) \rightarrow \widehat{\xi }\left( s\right) \alpha
_{s}^{u}\left( f\left( s^{-1}t\right) \right)
\end{equation*}%
from $G\times G$ to $E$ is continuous with compact support and the formula 
\begin{equation*}
t\rightarrow \left( \widehat{\xi }\cdot f\right) \left( t\right)
=\int\limits_{G}\widehat{\xi }\left( s\right) \alpha _{s}^{u}\left( f\left(
s^{-1}t\right) \right) ds,\text{ }t\in G
\end{equation*}%
defines an element $\widehat{\xi }\cdot f\in C_{c}(G,E)$. Thus, we have
defined a map from $C_{c}(G,A)\times C_{c}(G,E)\rightarrow C_{c}(G,E)$ by 
\begin{equation*}
\left( \widehat{\xi },f\right) \mapsto \widehat{\xi }\cdot f.
\end{equation*}%
It is not difficult to check that this map is $\mathbb{C}$ -linear with
respect to its variables. Let $\widehat{\xi }\in C_{c}(G,E)$ and $f,h\in $ $%
C_{c}(G,A)$. From 
\begin{eqnarray*}
\left( \widehat{\xi }\cdot \left( f\times h\right) \right) \left( t\right)
&=&\int\limits_{G}\widehat{\xi }\left( s\right) \alpha _{s}^{u}\left( \left(
f\times h\right) \left( s^{-1}t\right) \right) ds \\
&=&\int\limits_{G}\widehat{\xi }\left( s\right) \alpha _{s}^{u}\left(
\int\limits_{G}f\left( r\right) \alpha _{r}^{u}\left( h\left(
r^{-1}s^{-1}t\right) \right) dr\right) ds \\
&=&\int\limits_{G}\widehat{\xi }\left( s\right) \left( \int\limits_{G}\alpha
_{s}^{u}\left( f\left( s^{-1}g\right) \right) \alpha _{g}^{u}\left( h\left(
g^{-1}t\right) \right) dg\right) ds \\
&=&\int\limits_{G}\left( \int\limits_{G}\widehat{\xi }\left( s\right) \alpha
_{s}^{u}\left( f\left( s^{-1}g\right) \right) \alpha _{r}^{u}\left( h\left(
g^{-1}t\right) \right) dg\right) ds
\end{eqnarray*}%
and 
\begin{eqnarray*}
\left( \left( \widehat{\xi }\cdot f\right) \cdot h\right) \left( t\right)
&=&\int\limits_{G}\left( \widehat{\xi }\cdot f\right) \left( s\right) \alpha
_{s}^{u}\left( h\left( s^{-1}t\right) \right) ds \\
&=&\int\limits_{G}\left( \int\limits_{G}\widehat{\xi }\left( r\right) \alpha
_{r}^{u}\left( f\left( r^{-1}s\right) \right) dr\right) \alpha
_{s}^{u}\left( h\left( s^{-1}t\right) \right) ds \\
&=&\int\limits_{G}\left( \int\limits_{G}\widehat{\xi }\left( r\right) \alpha
_{r}^{u}\left( f\left( r^{-1}s\right) \right) \alpha _{s}^{u}\left( h\left(
s^{-1}t\right) \right) dr\right) ds \\
&=&\int\limits_{G}\left( \int\limits_{G}\widehat{\xi }\left( r\right) \alpha
_{r}^{u}\left( f\left( r^{-1}s\right) \right) \alpha _{s}^{u}\left( h\left(
s^{-1}t\right) \right) ds\right) dr
\end{eqnarray*}%
for all $t\in G$, we conclude that 
\begin{equation*}
\widehat{\xi }\cdot \left( f\times h\right) =\left( \widehat{\xi }\cdot
f\right) \cdot h.
\end{equation*}%
Therefore $C_{c}(G,E)$ has a structure of right $C_{c}(G,A)$-module.

Let $\widehat{\xi },\widehat{\eta }\in C_{c}(G,E)$. It is not difficult to
check that the function 
\begin{equation*}
\left( s,t\right) \rightarrow \alpha _{s^{-1}}^{u}\left( \left\langle 
\widehat{\xi }\left( s\right) ,\widehat{\eta }\left( st\right) \right\rangle
\right)
\end{equation*}%
from $G\times G$ to $A$ is continuous with compact support and the formula 
\begin{equation*}
t\rightarrow \int\limits_{G}\alpha _{s^{-1}}^{u}\left( \left\langle \widehat{%
\xi }\left( s\right) ,\widehat{\eta }\left( st\right) \right\rangle \right)
ds
\end{equation*}%
defines an element in $C_{c}(G,A)$. Thus, we have defined a $C_{c}(G,A)$%
-valued inner product $\left\langle \cdot ,\cdot \right\rangle
_{C_{c}(G,E)}:C_{c}(G,E)\times C_{c}(G,E)\rightarrow C_{c}(G,A)$ by 
\begin{equation*}
\left\langle \widehat{\xi },\widehat{\eta }\right\rangle
_{C_{c}(G,E)}(t)=\int\limits_{G}\alpha _{s^{-1}}^{u}\left( \left\langle 
\widehat{\xi }\left( s\right) ,\widehat{\eta }\left( st\right) \right\rangle
\right) ds\text{.}
\end{equation*}%
It is not difficult to check that the inner product defined above is $%
\mathbb{C}$-linear with respect to its second variable.

Let $f\in C_{c}(G,A),$ $\widehat{\xi },\widehat{\eta }\in C_{c}(G,E)$. From

\begin{eqnarray*}
\left\langle \widehat{\xi },\widehat{\eta }\cdot f\right\rangle
_{C_{c}(G,E)}(t) &=&\int\limits_{G}\alpha _{s^{-1}}^{u}\left( \left\langle 
\widehat{\xi }\left( s\right) ,\left( \widehat{\eta }\cdot f\right) \left(
st\right) \right\rangle \right) ds \\
&=&\int\limits_{G}\alpha _{s^{-1}}^{u}\left( \left\langle \widehat{\xi }%
\left( s\right) ,\int\limits_{G}\widehat{\eta }\left( g\right) \alpha
_{g}^{u}\left( f\left( g^{-1}st\right) \right) dg\right\rangle \right) ds \\
&=&\int\limits_{G}\alpha _{s^{-1}}^{u}\left( \int\limits_{G}\left(
\left\langle \widehat{\xi }\left( s\right) ,\widehat{\eta }\left( g\right)
\alpha _{g}^{u}\left( f\left( g^{-1}st\right) \right) \right\rangle \right)
dg\right) ds \\
&=&\int\limits_{G}\left( \int\limits_{G}\alpha _{s^{-1}}^{u}\left(
\left\langle \widehat{\xi }\left( s\right) ,\widehat{\eta }\left( g\right)
\alpha _{g}^{u}\left( f\left( g^{-1}st\right) \right) \right\rangle \right)
dg\right) ds \\
&=&\int\limits_{G}\left( \int\limits_{G}\alpha _{s^{-1}}^{u}\left(
\left\langle \widehat{\xi }\left( s\right) ,\widehat{\eta }\left( g\right)
\right\rangle \right) \alpha _{s^{-1}g}^{u}\left( f\left( g^{-1}st\right)
\right) dg\right) ds
\end{eqnarray*}%
and%
\begin{eqnarray*}
\left( \left\langle \widehat{\xi },\widehat{\eta }\right\rangle
_{C_{c}(G,E)}\times f\right) (t) &=&\int\limits_{G}\left\langle \widehat{\xi 
},\widehat{\eta }\right\rangle _{C_{c}(G,A)}(s)\alpha _{s}^{u}\left( f\left(
s^{-1}t\right) \right) ds \\
&=&\int\limits_{G}\left( \int\limits_{G}\alpha _{r^{-1}}^{u}\left(
\left\langle \widehat{\xi }\left( r\right) ,\widehat{\eta }\left( rs\right)
\right\rangle \right) dr\right) \alpha _{s}^{u}\left( f\left( s^{-1}t\right)
\right) ds \\
&=&\int\limits_{G}\left( \int\limits_{G}\alpha _{r^{-1}}^{u}\left(
\left\langle \widehat{\xi }\left( r\right) ,\widehat{\eta }\left( rs\right)
\right\rangle \right) \alpha _{s}^{u}\left( f\left( s^{-1}t\right) \right)
dr\right) ds \\
&&\text{Fubini's Theorem } \\
&=&\int\limits_{G}\left( \int\limits_{G}\alpha _{r^{-1}}^{u}\left(
\left\langle \widehat{\xi }\left( r\right) ,\widehat{\eta }\left( rs\right)
\right\rangle \right) \alpha _{s}^{u}\left( f\left( s^{-1}t\right) \right)
ds\right) dr \\
&=&\int\limits_{G}\left( \int\limits_{G}\alpha _{r^{-1}}^{u}\left(
\left\langle \widehat{\xi }\left( r\right) ,\widehat{\eta }\left( g\right)
\right\rangle \right) \alpha _{r^{-1}g}^{u}\left( f\left( g^{-1}rt\right)
\right) dg\right) dr
\end{eqnarray*}%
for all $t\in G$, we deduce that 
\begin{equation*}
\left\langle \widehat{\xi },\widehat{\eta }\cdot f\right\rangle
_{C_{c}(G,E)}=\left\langle \widehat{\xi },\widehat{\eta }\right\rangle
_{C_{c}(G,E)}\times f\text{.}
\end{equation*}

From 
\begin{eqnarray*}
\left( \left\langle \widehat{\xi },\widehat{\eta }\right\rangle
_{C_{c}(G,E)}\right) ^{\#}(t) &=&\Delta \left( t\right) ^{-1}\alpha
_{t}^{u}\left( \left( \left\langle \widehat{\xi },\widehat{\eta }%
\right\rangle _{C_{c}(G,E)}(t^{-1})\right) ^{\ast }\right) \\
&=&\Delta \left( t\right) ^{-1}\alpha _{t}^{u}\left( \left(
\int\limits_{G}\alpha _{s^{-1}}^{u}\left( \left\langle \widehat{\xi }\left(
s\right) ,\widehat{\eta }\left( st^{-1}\right) \right\rangle \right)
ds\right) ^{\ast }\right) \\
&=&\Delta \left( t\right) ^{-1}\alpha _{t}^{u}\left( \int\limits_{G}\alpha
_{s^{-1}}^{u}\left( \left\langle \widehat{\eta }\left( st^{-1}\right) ,%
\widehat{\xi }\left( s\right) \right\rangle \right) ds\right) \\
&=&\Delta \left( t\right) ^{-1}\int\limits_{G}\alpha _{ts^{-1}}^{u}\left(
\left\langle \widehat{\eta }\left( st^{-1}\right) ,\widehat{\xi }\left(
s\right) \right\rangle \right) ds \\
&=&\Delta \left( t\right) ^{-1}\int\limits_{G}\alpha _{g^{-1}}^{u}\left(
\left\langle \widehat{\eta }\left( g\right) ,\widehat{\xi }\left( gt\right)
\right\rangle \right) \Delta \left( t\right) dg \\
&=&\left\langle \widehat{\eta },\widehat{\xi }\right\rangle _{C_{c}(G,E)}(t)
\end{eqnarray*}%
for all $t\in G$, we deduce that 
\begin{equation*}
\left( \left\langle \widehat{\xi },\widehat{\eta }\right\rangle
_{C_{c}(G,E)}\right) ^{\#}=\left\langle \widehat{\eta },\widehat{\xi }%
\right\rangle _{C_{c}(G,E)}.
\end{equation*}

Let $\widehat{\xi }\in C_{c}(G,E)$ and $p\in S(A).$ Then 
\begin{eqnarray*}
\left( \pi _{p}^{A}\circ \left\langle \widehat{\xi },\widehat{\xi }%
\right\rangle _{C_{c}(G,E)}\right) (t) &=&\pi _{p}^{A}\left(
\int\limits_{G}\alpha _{s^{-1}}^{u}\left( \left\langle \widehat{\xi }\left(
s\right) ,\widehat{\xi }\left( st\right) \right\rangle \right) ds\right) \\
&=&\int\limits_{G}\pi _{p}^{A}\left( \alpha _{s^{-1}}^{u}\left( \left\langle 
\widehat{\xi }\left( s\right) ,\widehat{\xi }\left( st\right) \right\rangle
\right) \right) ds \\
&=&\int\limits_{G}\alpha _{s^{-1}}^{u^{p}}\left( \left\langle \left( \sigma
_{p}^{E}\circ \widehat{\xi }\right) \left( s\right) ,\left( \sigma
_{p}^{E}\circ \widehat{\xi }\right) \left( st\right) \right\rangle \right) ds
\\
&=&\left( \left\langle \sigma _{p}^{E}\circ \widehat{\xi },\sigma
_{p}^{E}\circ \widehat{\xi }\right\rangle _{C_{c}(G,E_{p})}\right) \left(
t\right)
\end{eqnarray*}%
for all $t\in G$. From this fact and \cite[Remark pp. 300 ]{2}, we deduce
that $\pi _{p}^{A}\circ \left\langle \widehat{\xi },\widehat{\xi }%
\right\rangle _{C_{c}(G,E)}$ is a positive element in $C_{c}(G,A_{p})$ for
all $p\in S(A)$. Therefore $\left\langle \widehat{\xi },\widehat{\xi }%
\right\rangle _{C_{c}(G,E)}$ is a positive element in $C_{c}(G,A)$.

Let $\widehat{\xi }\in C_{c}(G,A)$ such that $\left\langle \widehat{\xi },%
\widehat{\xi }\right\rangle _{C_{c}(G,E)}=0.$ Then $\pi _{p}^{A}\circ
\left\langle \widehat{\xi },\widehat{\xi }\right\rangle _{C_{c}(G,E)}=0$ for
all $p\in S(A),$ and since $\pi _{p}^{A}\circ \left\langle \widehat{\xi },%
\widehat{\xi }\right\rangle _{C_{c}(G,E)}=\left\langle \sigma _{p}^{E}\circ 
\widehat{\xi },\sigma _{p}^{E}\circ \widehat{\xi }\right\rangle
_{C_{c}(G,E_{p})}$ for all $p\in S(A),$ we have $\left\langle \sigma
_{p}^{E}\circ \widehat{\xi },\sigma _{p}^{E}\circ \widehat{\xi }%
\right\rangle _{C_{c}(G,E_{p})}=0$ for all $p\in S(A)$. From this fact and 
\cite[Remark pp. 300 ]{2}, we deduce that $\sigma _{p}^{E}\circ \widehat{\xi 
}=0$ for all $p\in S(A)$ and so $\widehat{\xi }=0$. Therefore $\left\langle 
\widehat{\xi },\widehat{\xi }\right\rangle _{C_{c}(G,E)}=0$ if and only of $%
\widehat{\xi }=0$.

Thus we showed that $C_{c}(G,E)$ is a right $C_{c}(G,A)$-module equipped
with a $C_{c}(G,A)$-valued inner product which is $\mathbb{C}$-and $%
C_{c}(G,A)$-linear in its second variable and verify the relations $(1),(2)$
and $(3)$ of Definition 2.1.

Let $E^{G}$ be the completion of $C_{c}(G,E)$ with respect to the topology
determined by the inner product if we consider on $C_{c}(G,A)$ the topology
determined by the family of $C^{\ast }$-seminorms $\{n_{p}\}_{p\in S(A)}$.
Then $E^{G}$ is a Hilbert $G\times _{\alpha ^{u}}A$-module \cite[Remark
1.2.10]{10}.

\begin{lemma}
Let $p\in S(A)$. Then the Hilbert $G\times _{\alpha ^{u^{p}}}A_{p}$-modules $%
\left( E^{G}\right) _{p}$ and $E_{p}^{G}$ are unitarily equivalent.
\end{lemma}

\begin{proof}
Let $\widehat{\xi }\in C_{c}(G,E)$. Then 
\begin{eqnarray*}
\overline{n_{p}}\left( \widehat{\xi }\right) ^{2} &=&n_{p}\left(
\left\langle \widehat{\xi },\widehat{\xi }\right\rangle _{C_{c}(G,E)}\right) 
\\
&=&\sup \{\left\Vert \varphi \left( \left\langle \widehat{\xi },\widehat{\xi 
}\right\rangle _{C_{c}(G,E)}\right) \right\Vert ;\varphi \in \mathcal{R}%
_{p}(G\times _{\alpha ^{u}}A)\} \\
&=&\sup \{\left\Vert \varphi _{p}\left( \left\langle \sigma _{p}^{E}\circ 
\widehat{\xi },\sigma _{p}^{E}\circ \widehat{\xi }\right\rangle
_{C_{c}(G,E_{p})}\right) \right\Vert ;\varphi _{p}\in \mathcal{R}\left(
(G\times _{\alpha ^{u}}A)_{p}\right) \} \\
&=&\left\Vert \sigma _{p}^{E}\circ \widehat{\xi }\right\Vert
_{E_{p}^{G}}^{2}.
\end{eqnarray*}%
Thus we can define a linear map $U_{p}:C_{c}(G,E)/\ker \left( \overline{n_{p}%
}|_{C_{c}(G,E)}\right) $ $\rightarrow C_{c}(G,E_{p})$ by 
\begin{equation*}
U_{p}\left( \widehat{\xi }+\ker \left( \overline{n_{p}}|_{C_{c}(G,E)}\right)
\right) =\sigma _{p}^{E}\circ \widehat{\xi }.
\end{equation*}%
Moreover, 
\begin{eqnarray*}
&&\left\langle U_{p}\left( \widehat{\xi }+\ker \left( \overline{n_{p}}%
|_{C_{c}(G,E)}\right) \right) ,U_{p}\left( \widehat{\xi }+\ker \left( 
\overline{n_{p}}|_{C_{c}(G,E)}\right) \right) \right\rangle _{C_{c}(G,E_{p})}
\\
&=&\left\langle \sigma _{p}^{E}\circ \widehat{\xi },\sigma _{p}^{E}\circ 
\widehat{\xi }\right\rangle _{C_{c}(G,E_{p})} \\
&=&\pi _{p}^{A}\circ \left\langle \widehat{\xi },\widehat{\xi }\right\rangle
_{C_{c}(G,E)} \\
&=&\left\langle \widehat{\xi }+\ker \left( \overline{n_{p}}%
|_{C_{c}(G,E)}\right) ,\widehat{\xi }+\ker \left( \overline{n_{p}}%
|_{C_{c}(G,E)}\right) \right\rangle _{\left( C_{c}(G,E)\right) _{p}}
\end{eqnarray*}%
where $\left( C_{c}(G,E)\right) _{p}=C_{c}(G,E)/\ker \left( \overline{n_{p}}%
|_{C_{c}(G,E)}\right) $. For $\xi \in E$ and $f\in C_{c}(G)$, the map $%
s\rightarrow $ $\widehat{\xi _{f}},$ where $\widehat{\xi _{f}}(s)=f(s)\xi ,$
defines an element in $C_{c}(G,E)$ and 
\begin{equation*}
U_{p}\left( \widehat{\xi _{f}}+\ker \left( \overline{n_{p}}%
|_{C_{c}(G,E)}\right) \right) \left( t\right) =f(t)\sigma _{p}^{E}\left( \xi
\right) =\widehat{\sigma _{p}^{E}\left( \xi \right) _{f}}.
\end{equation*}%
From these facts, \cite[Theorem 3.5 ]{12} and taking into account that the
vector space $\left( C_{c}(G,E)\right) _{p}$ is dense in $\left(
E^{G}\right) _{p}$ and $E_{p}\otimes _{\text{alg}}C_{c}(G)$ is dense in $%
E_{p}^{G}$, we deduce that $U_{p}$ extends to a unitary operator from $%
\left( E^{G}\right) _{p}$ to $E_{p}^{G}$. Therefore the Hilbert $G\times
_{\alpha ^{u^{p}}}A_{p}$-modules $\left( E^{G}\right) _{p}$ and $E_{p}^{G}$
are unitarily equivalent.
\end{proof}

\begin{corollary}
The Hilbert $G\times _{\alpha ^{u}}A$-modules $E^{G}$ and $\lim\limits_{%
\underset{p\in S(A)}{\leftarrow }}E_{p}^{G}$ are unitarily equivalent and
the locally $C^{\ast }$-algebras $K\left( E^{G}\right) $ and $\lim\limits_{%
\underset{p\in S(A)}{\leftarrow }}K(E_{p}^{G})$ are isomorphic.
\end{corollary}

\begin{proof}
If $U_{p}$ is the unitary operator from $\left( E^{G}\right) _{p}$ to $%
E_{p}^{G}$ constructed in the proof of Lemma 5.3, for each $p\in S(A),$ then 
$\left( U_{p}\right) _{p}$ is an inverse system of unitary operators and so
the Hilbert $G\times _{\alpha ^{u}}A$-module $E^{G}$ and $\lim\limits_{%
\underset{p\in S(A)}{\leftarrow }}E_{p}^{G}$ are unitarily equivalent. By 
\cite[Proposition 4.7]{14}, the locally $C^{\ast }$-algebras $K\left(
E^{G}\right) $ and $\lim\limits_{\underset{p\in S(A)}{\leftarrow }%
}K(E_{p}^{G})$ are isomorphic.
\end{proof}

\begin{definition}
(\cite{6}) Two locally $C^{\ast }$-algebras $A$ and $B$ are strongly Morita
equivalent if there is a full Hilbert $A$-module $E$ such that the locally $%
C^{\ast }$-algebras $K(E)$ and $B$ are isomorphic.
\end{definition}

\begin{theorem}
Let $G$ be a locally compact group and let $\alpha $ and $\beta $ be two
continuous inverse limit actions of $G$ on the locally $C^{\ast }$-algebras $%
A$ and $B$. If the actions $\alpha $ and $\beta $ are strongly Morita
equivalent, then the crossed products $G\times _{\alpha }A$ and $G\times
_{\beta }B$ are strongly Morita equivalent.
\end{theorem}

\begin{proof}
Suppose that $\alpha \backsim _{E,u}^{s}\beta $. Then, by Proposition 4.4,
the locally $C^{\ast }$-algebras $G\times _{\alpha }A$ and $G\times _{\alpha
^{u}}A$ are isomorphic and so they are strongly Morita equivalent as well as
the locally $C^{\ast }$-algebras $G\times _{\beta }B$ and $G\times _{\beta
^{u}}K(E)$. To proof the theorem it is sufficient to prove that the locally $%
C^{\ast }$-algebras $G\times _{\alpha ^{u}}A$ and $G\times _{\beta ^{u}}K(E)$
are strongly Morita equivalent.

For each $p\in S(A)$, a simple calculus shows that $\alpha ^{u^{p}}\backsim
_{E_{p},u^{p}}^{s}\beta ^{u^{p}}$ and by \cite[Remark pp. 300]{2}, the
Hilbert $G\times _{\alpha ^{u^{p}}}A_{p}$-module $E_{p}^{G}$ implements a
strong Morita equivalence between $G\times _{\alpha ^{u^{p}}}A_{p}$ and $%
G\times _{\beta ^{u^{p}}}K(E_{p})$. Moreover, the linear map $\Phi
_{p}:\Theta (E_{p}^{G})\rightarrow C_{c}(G,K(E_{p}))$ defined by 
\begin{equation*}
\Phi _{p}\left( \theta _{\widehat{\xi _{p}},\widehat{\eta _{p}}}\right)
(t)=\int_{G}\theta _{\widehat{\xi _{p}}(s),\Delta (t^{-1}s)u_{t}^{p}\left( 
\widehat{\eta _{p}}(t^{-1}s)\right) }ds
\end{equation*}%
extends to an isomorphism of $C^{\ast }$-algebras from $K(E_{p}^{G})$ to $%
G\times _{\beta ^{u^{p}}}K(E_{p})$ \cite[Remark pp. 300]{2}.

Let $E^{G}$ be the Hilbert $G\times _{\alpha ^{u}}A$-module constructed
above. From Corollary 5.4 and taking into account that for each $p\in S(A)$,
the Hilbert $G\times _{\alpha ^{u^{p}}}A_{p}$-module $E_{p}^{G}$ is full, we
conclude that $E^{G}$ is full. It is not difficult to check that 
\begin{equation*}
\pi _{pq}^{G\times _{\beta ^{u}}K(E)}\circ \Phi _{p}=\Phi _{q}\circ \left(
\pi _{pq}^{K(E^{G})}\right) _{\ast }
\end{equation*}%
for all \ $p,q\in S(A)$ with $p\geq q$, where \ $\{\pi _{pq}^{G\times
_{\beta ^{u}}K(E)}\}_{p,q\in S(A),p\geq q}$ and $\{\left( \pi
_{pq}^{K(E^{G})}\right) _{\ast }$ $\}_{p,q\in S(A),p\geq q}$ are the
connecting maps of the inverse systems of $C^{\ast }$-algebras $\{G\times
_{\beta ^{u^{p}}}K(E_{p})\}_{p\in S(A)}$ respectively $\{K(E_{p}^{G})\}_{p%
\in S(A)}$. Therefore $\left( \Phi _{p}\right) _{p}$ is an inverse system of 
$C^{\ast }$-isomorphisms, and then the locally $C^{\ast }$-algebras $%
\lim\limits_{\underset{p\in S(A)}{\leftarrow }}K(E_{p}^{G})$ and $%
\lim\limits_{\underset{p\in S(A)}{\leftarrow }}G\times _{\beta
^{u^{p}}}K(E_{p})$ are isomorphic. From this fact, we deduce that the
locally $C^{\ast }$-algebras $K\left( E^{G}\right) $ and $G\times _{\beta
^{u}}K(E)$ are isomorphic and so the locally $C^{\ast }$-algebras $G\times
_{\alpha ^{u}}A$ and $G\times _{\beta ^{u}}K(E)$ are strongly Morita
equivalent.
\end{proof}

\begin{corollary}
Let $G$ be a compact group and let $\alpha $ and $\beta $ be two continuous
actions of $G$ on the locally $C^{\ast }$-algebras $A$ and $B$ such that the
maps $\left( g,a\right) \rightarrow \alpha _{g}(a)$ from $G\times A$ to $A$
and $\left( g,b\right) \rightarrow \beta _{g}(b)$ from $G\times B$ to $B$
are jointly continuous. If the actions $\alpha $ and $\beta $ are strongly
Morita equivalent, then the crossed products $G\times _{\alpha }A$ and $%
G\times _{\beta }B$ are strongly Morita equivalent.
\end{corollary}

\begin{proof}
Since the group $G$ is compact, the actions $\alpha $ and $\beta $ of $G$ on 
$A$ and $B$ are continuous inverse limit actions \cite{15} and apply Theorem
5.6.
\end{proof}

Department of Mathematics, Faculty of Chemistry, University of Bucharest,
Bd. Regina Elisabeta nr. 4-12, Bucharest, Romania

mjoita@fmi.unibuc.ro

\end{document}